\documentclass[a4paper,12pt]{amsart}
\usepackage{amsfonts,amsmath,amsthm,amssymb}
\usepackage{paralist}
\usepackage{latexsym}
\usepackage{graphics}
\newtheorem{theorem}{Theorem}[section]

\newtheorem{lemma}[theorem]{Lemma}

\newtheorem{proposition}[theorem]{Proposition}

\theoremstyle{definition}

\theoremstyle{remark}
\newtheorem{rem}[theorem]{Remark}
\theoremstyle{remark}

\newcommand{\beql}[1]{\begin{equation}\label{#1}}
\newcommand{\eeq}{\end{equation}}

\begin{document}
\newcommand{\BH}{\mathrm{BH}}


\title[The quaternary complex Hadamard matrices of orders $10,12$, and
  $14$]{The quaternary complex Hadamard matrices\newline of orders
  $10,12$, and $14$}

\author{Pekka H.J. Lampio, Ferenc Sz\"oll\H{o}si and Patric R.J. \"Osterg\aa rd}

\date{February 2012, preprint}

\address{P.H.J. L.: Department of Communications and Networking,
  Aalto University School of Electrical Engineering,
  P.O. Box 13000, 00076 Aalto, Finland.
}\email{pekka.lampio@aalto.fi}

\address{F. Sz.: Department of Mathematics and its
  Applications, Central European University, H-1051, N\'ador u. 9,
  Budapest, Hungary}\email{szoferi@gmail.com}

\address{P.R.J. \"O.: Department of Communications and Networking,
  Aalto University School of Electrical Engineering,
  P.O. Box 13000, 00076 Aalto, Finland.
}\email{patric.ostergard@aalto.fi}

\thanks{P.H.J. L. was supported by the Academy of Finland, Grants No. 132122.\\
  F. Sz. was supported by the Hungarian National Research Fund OTKA K-77748.\\
  P.R.J. \"O. was supported by the Academy of Finland, Grants
  No. 130142 and No. 132122.\\}

\begin{abstract}
A complete classification of quaternary complex Hadamard matrices of
orders $10, 12$ and $14$ is given, and a new parametrization scheme
for obtaining new examples of affine parametric families of complex
Hadamard matrices is provided. On the one hand, it is proven that all
$10\times 10$ and $12\times 12$ quaternary complex Hadamard matrices
belong to some parametric family, but on the other hand, it is shown
by exhibiting an isolated $14\times 14$ matrix that there cannot be a
general method for introducing parameters into these types of matrices.
\end{abstract}

\maketitle

{\bf 2010 Mathematics Subject Classification.} Primary 05B20,
secondary 15B34.

{\bf Keywords and phrases.} {\it Butson-type Hadamard matrix,
  Classification, Complex Hadamard matrix}.

\section{Introduction}

A complex Hadamard matrix $H$ of order $n$ is an $n\times n$ matrix
with unimodular entries satisfying $HH^\ast=nI$ where $\ast$ is the
conjugate transpose, $I$ is the identity matrix of order $n$, and
where unimodular means that the entries are complex numbers that lie
on the unit circle. In other words, any two distinct rows (or columns)
of $H$ are complex orthogonal. Complex Hadamard matrices have various
applications in mathematics ranging from coding \cite{hlt} and
operator theory \cite{popa} to harmonic analysis \cite{MMM}. They also
play a crucial r\^ole in quantum information theory, for construction of
quantum teleportation and dense coding schemes \cite{wer}, and they
are strongly related to mutually unbiased bases (MUBs) \cite{Z}. In
this paper we are primarily concerned with $n\times n$ Butson-type
Hadamard matrices \cite{Bu}, denoted by $\BH(q,n)$, which are complex
Hadamard matrices composed of $q$th roots of unity. The notations
$\BH(2,n)$ and $\BH(4,n)$ correspond to the real Hadamard matrices and
quaternary complex Hadamard matrices, respectively.

Recently there has been a renewed interest in $\BH(q,n)$ matrices. In
particular, they have been used as starting-point matrices for
constructing parametric families of complex Hadamard matrices for
various $q$ and $n$.  In a series of papers Di\c{t}\u{a} \cite{dita3},
Sz\"oll\H{o}si \cite{SZF1}, and Tadej and \.Zyczkowski \cite{karol}
obtained new, previously unknown parametric families of complex
Hadamard matrices from $\BH(q,n)$ matrices. The parametrization of a
$\BH(q,n)$ matrix is an operation where some of the matrix entries are
replaced by infinitely differentiable, or smooth, functions mapping a
real-valued or complex-valued argument vector into the set of
unimodular complex numbers. Any assignment of values to the parameters
yields a complex Hadamard matrix and some assignments produce
$\BH(q,n)$ matrices. The parametrization of the $\BH(q,n)$ matrices
makes it possible to escape equivalence classes and therefore to
collect many inequivalent matrices into a single parametric family.
Also the switching operation \cite{WO}, a well-known technique in
design theory, has this property.

The existence of $\BH(q,n)$ matrices is wide open in general. Even the
simplest case for $q=2$ is undecided, as the famous Hadamard
conjecture, stating that $\BH(2,4k)$ matrices exist for every positive
integer $k$, remains elusive despite continuous efforts
\cite{CHK}. Real Hadamard matrices, or $\BH(2,n)$ matrices, are
completely classified up to $n=28$, and there has been some promising
advance in enumerating the case $n=32$ very recently (see \cite{HK}
and the references therein). For other values of $q$ we have some
constructions \cite{Bu} and some non-existence results
\cite{winterhof}. Harada, Lam and Tonchev classified all $16\times 16$
generalized Hadamard matrices over groups of order $4$ and obtained
new examples of $\BH(4,16)$ matrices \cite{hlt}. This particular
result motivated the authors to investigate the existence of
$\BH(4,n)$ matrices in more general and to start enumerating and
classifying them for small $n$. The census of quaternary complex
Hadamard matrices up to order $8$ were carried out in \cite{SZF4}. The
aim of this paper is to continue that work, and give a complete
classification of the $\BH(4,n)$ matrices up to orders $n=14$.

This work has two major parts: first, we classify all $\BH(q,n)$
matrices using computer-aided methods, and secondly, we define
parametric families of complex Hadamard matrices by introducing
parameters to the $\BH(q,n)$ matrices. Parametric families offer a
compact way of representing a large number of $\BH(q,n)$ matrices, and
they also yield information about complex Hadamard matrices.

In the previous work \cite{SZF4} the authors collected all $\BH(4,n)$
matrices contained in various parametric families from the existing
literature and then confirmed with an exhaustive computer search that
those are the only examples. Here the situation is exactly the
opposite, as it turns out that almost all $\BH(4,n)$ matrices of
orders $n=12$ and $14$ found in the current work by the computer
search are previously unknown. Therefore we are facing the inverse
problem: we need to encode a given collection of $\BH(4,n)$ matrices
by parametric families.

The outline of the paper is as follows. In Section \ref{preliminaries}
we give some basic definitions and results that are used in later
sections.  Then in Section \ref{ch2} we recall the computer-aided
classification method of difference matrices from \cite{lampio} and
highlight the main differences and necessary modifications required
for our purposes. In Section \ref{ch3} we introduce a new method for
parametrizing complex Hadamard matrices and recall some relevant
materials from the existing literature. Finally, in Sections
\ref{ch4}, \ref{ch5}, and \ref{ch6} we present the classification of
$\BH(4,n)$ matrices for $n=10,12$, and $14$, respectively. To improve
the readability of our paper a rather long list of complex Hadamard
matrices has been moved from the main text to Appendix \ref{APPA}. The
matrices the authors obtained are also available in an electronic
format in the on-line repository \cite{web} as a
supplement. Throughout this paper we adopt the notations from
\cite{karol} for known complex Hadamard matrices, such as
$D_{10}^{(3)}(a,b,c)$, etc.

\section{Preliminaries}\label{preliminaries}

Two $\BH(q,n)$ matrices are equivalent if the first matrix can be
transformed into the second one by permuting the order of its rows and
columns and multiplying a row or a column by some $q$th root of
unity. The automorphism group of a $\BH(q,n)$ matrix $H$ is the group
of pairs of monomial matrices $(P,Q)$ such that $H=PHQ$; a monomial
matrix is an $n\times n$ matrix having a single nonzero entry in each
row and column, these nonzero entries being complex $q$th roots of
unity. Note that the automorphism group depends on the choice of $q$.

We employ invariants for finding the equivalence class of a $\BH(q,n)$
matrix in a computationally efficient way. The determinant of a $k
\times k$ submatrix of matrix $H$ is called a vanishing minor, or a
zero minor, of $H$ if the determinant is zero. Throughout this paper
we shall repeatedly use the following

\begin{lemma}
Let $H$ be a complex Hadamard matrix of order $n$ and $2\leq k\leq
n-2$ be an integral number. Then the number of $k\times k$ vanishing
minors of $H$ is invariant, up to equivalence.
\end{lemma}

The number of $k\times k$ vanishing minors is just a special case of a
more powerful invariant, the fingerprint \cite{SZF1}, but it is sufficient
for our purposes. 

The $q$-rank, or $\mathbb{Z}_q$-rank, of an integral matrix $L$ of
order $n$ is the smallest positive integer $r$ such that there are
integral matrices $S$ and $T$ of orders $n\times r$ and $r\times n$
respectively, such that $ST\equiv L$ (mod $q$). The
$\mathbb{Z}_q$-rank of a $\BH(q,n)$ matrix $H$ is the
$\mathbb{Z}_q$-rank of the $(0,1,\hdots, q-1)$-matrix $L$ for which
$H=\mathrm{EXP}\left(\frac{2\pi\mathbf{i}}{q}L\right)$, where
$\mathrm{EXP}$ denotes the entry-wise exponentiation function. It is
easy to see that the $\mathbb{Z}_q$-rank is, again, invariant up to
equivalence. A class of $\BH(q,n)$ matrices with small
$\mathbb{Z}_q$-rank has some interesting applications in harmonic
analysis \cite{MMM}.

As it seems that a complex Hadamard matrix $H$ shares every important
property with its hermitian $H^\ast$, complex conjugate $\overline{H}$
and transpose $H^T$, we introduce the concept of ACT-equivalence.  Two
complex Hadamard matrices $H$ and $K$ are called
$\mathrm{ACT}$-equivalent, if $H$ is equivalent to at least one of
$K$, $K^\ast$, $\overline{K}$ or $K^T$.  The concept of this refined
equivalence simplifies the presentation of our results as we can avoid
unnecessary repetitions in our summarizing tables \cite{web}.

\section{Constructing Butson-type Hadamard matrices}\label{ch2}

The computer-aided methods employed in the classification of
Butson-type Hadamard matrices in this work are very similar to the
methods used for the classification of difference matrices over cyclic
groups \cite{lampio}. Therefore, we give here only a summary of the
relevant ideas and describe in detail only the points where this work
differs from the work done on difference matrices.

We perform an exhaustive computer search of the space of all
Butson-type Hadamard matrices $\BH(q,n)$ with the given parameters
$q=4$ and $n=10,12,14$. Because the size of the search space grows
very quickly as the parameters are increased, we have to prune most of
the search space and we do this by considering only inequivalent
matrices.

An $m \times n$ matrix over $q$th roots of unity is said to be a
candidate $\BH(q,n)$ matrix, denoted by $\BH(q,m,n)$, if the rows are
orthogonal as in a $\BH(q, n)$ matrix, but the number of rows $m \leq
n$. A $\BH(q,m,n)$ matrix is dephased (or normalized) if the first
row and the first column contain only the value $1$. Since every
$\BH(q,m,n)$ matrix is equivalent to a dephased one, only these
matrices are considered in the computer search.

The dephased candidate $\BH(q,m,n)$ matrices are organized into a tree
where the child nodes of a $\BH(q,m,n)$ are all $\BH(q,m+1,n)$ that
are obtained by appending a row to the node,
cf.\ \cite[Definition~2.6]{lampio}. An exhaustive search of this tree
is performed by weak canonical augmentation
\cite[Section~4.2.3]{kaski}, which has the advantage over the simpler
breadth-first search used in \cite{lampio} that it can be performed
easily in parallel.

Consider a tree of dephased candidate $\BH(q,m,n)$ matrices. Let $p(X)$
denote the parent of the node $X$ in the tree. Every node
$X$ in the tree has a finite sequence of ancestors from which it
has been constructed by appending a row to it:
\begin{equation}\label{x_seq}
  X,p(X),p(p(X)),p(p(p(X))), \dots
\end{equation}

In the tree each $X$ occurs only once but there are typically many
nodes $Y$ with $X \cong Y$. Such a node $Y$ has also a sequence of
ancestors from which it has been constructed:
\begin{equation}\label{y_seq}
  Y,p(Y),p(p(Y)),p(p(p(Y))), \dots
\end{equation}

Even though $X \cong Y$ the ancestor sequences (\ref{x_seq}) and
(\ref{y_seq}) need not consist of the same nodes up to equivalence. This
means that the sequences (\ref{x_seq}) and (\ref{y_seq}) can be distinct
on the level of equivalence classes even if $X \cong Y$. The main idea
is now to exploit such differences among equivalent nodes in rejecting
equivalent nodes.

Let $T_{nr}$ denote the set of all non-root nodes in the search
tree $T$. Associate with every object $X \in T_{nr}$ a weak canonical
parent $w(X) \in T$ such that the following property holds:
\begin{equation}\label{weak_parent}
  \mbox{for all} \quad X,Y \in T_{nr} \quad \mbox{it holds that}
  \quad X \cong Y \quad \mbox{implies} \quad w(X) \cong w(Y).
\end{equation}

The function $w$ defines for every non-root node $X$ a sequence of
objects analogous to (\ref{x_seq}):
\begin{equation}\label{w_seq}
  X,w(X),w(w(X)),w(w(w(X))), \dots
\end{equation}

Because of (\ref{weak_parent}), any two equivalent objects have
identical sequences (\ref{w_seq}) on the level of equivalence classes of
matrices. When the search tree is traversed in depth-first order, a
node $X$ and the subtree rooted at it is considered only if it has
been constructed in the canonical way specified by (\ref{w_seq}); that
is, every matrix in the ancestor sequence (\ref{x_seq}) should be
equivalent to the matrix in the corresponding position in the sequence
(\ref{w_seq}). By (\ref{weak_parent}) we obtain
$$ p(X) \cong w(X) \cong w(Y) \cong p(Y) $$

In other words, equivalent matrices generated by weak canonical
augmentation have equivalent parent matrices, and assuming that the
same holds for the parents, this implies that equivalent matrices must
be siblings in the search tree. This reduces the size of the search
tree dramatically.

The problem of checking equivalence of $\BH(q,m,n)$ matrices is solved
by transforming it into a corresponding graph isomorphism problem in
exactly the same way as was done with difference matrices over cyclic
groups in \cite[Section~3]{lampio}.  Each $\BH(q,m,n)$ matrix is
mapped to a directed graph, called the equivalence graph of the
matrix. Two $\BH(q,m,n)$ matrices are equivalent if and only if their
equivalence graphs are isomorphic graphs. The definition of an
equivalence graph is obtained from the definition of a difference
matrix graph \cite[Definition~3.2]{lampio} by considering $q$th roots
of unity as a cyclic group where the operation is complex
multiplication.

The software package \textsc{Nauty} \cite{nauty} is used for checking
the isomorphism of equivalence graphs. For each equivalence graph
\textsc{Nauty} calculates a graph, called the canonical graph, which
has the property that two equivalence graphs are isomorphic if and
only if they have the same canonical graph. This implies that two
$\BH(q,m,n)$ matrices are equivalent if and only if the canonical
graphs of the equivalence graphs of the matrices are the same. We
define the canonical graph of a matrix $X$ as the canonical graph of
the equivalence graph of $X$.

Transforming a matrix to a graph yields also a weak canonical function
$w$ having the property (\ref{weak_parent}). Let $P(X)$ be the set of
matrices obtained by removing a row from a matrix $X \in T_{nr}$ and
let $\leq_g$ denote a total order on the set of canonical graphs. We
define $w(X)$ as the matrix $Y \in P(X)$ which has the smallest
canonical graph under $\leq_g$ among the canonical graphs of matrices
in $P(X)$.

The definition of equivalence of $\BH(q,n)$ matrices given in Section
\ref{preliminaries} agrees with the equivalence relation $\cong^*$
defined in \cite[Section~2]{lampio} and the equivalence inducing group
$E^*$ in \cite[Section~5]{lampio}.

To get confidence in the computational results, we perform a
consistency check by counting the number of candidate and complete
$\BH(q,n)$ matrices in two different ways. In the first method the
sizes of the equivalence classes of $\BH(q,m,n)$ matrices are obtained
from inequivalent $\BH(q,m,n)$ matrices and the automorphism groups of
their canonical graphs via the orbit-stabilizer theorem.  The second
method determines the number of $\BH(q,m,n)$ matrices from the number
of $\BH(q,m - 1,n)$ matrices and number of ways these smaller matrices
can be augmented to yield a $\BH(q,m,n)$ matrix.  This double-counting
method is the same as in \cite[Section~5]{lampio} except that the
analytical formula for $2 \times c$ difference matrices and the
reduction used for $3 \times c$ difference matrices are not valid for
$\BH(q,2,n)$ and $\BH(q,3,n)$ matrices, respectively. With $\BH(q,n)$
matrices the double-counting starts one level earlier with
$\BH(q,1,n)$ matrices, which all belong to the same equivalent class of
size $q^n$.

The exhaustive computer search yields all inequivalent $\BH(q,n)$ matrices and
the automorphism group for each matrix.

\section{Parametrizing complex Hadamard matrices}\label{ch3}

Once a complete set of inequivalent $\BH(4,n)$ matrices were obtained
the authors investigated whether these matrices lead to parametric
families of complex Hadamard matrices. We recall the following basic
method to introduce free parameters into complex Hadamard matrices of
even order \cite[Lemma~3.4]{SZF1}.

\begin{lemma}\label{spl1}
Let $H$ be a dephased complex Hadamard matrix of order $n$ and suppose
that there exist a pair of rows in $H$, say $u$ and $v$, such that for
every $i=1,2,\hdots,n$, $u_i^2=v_i^2$. Then for all such $i$ for which
$u_i+v_i=0$ replace $u_i$ with $\alpha u_i$ and $v_i$ with $\alpha
v_i$, where $\alpha$ is a unimodular complex number to obtain a
one-parameter family of complex Hadamard matrices $H(\alpha)$.
\end{lemma}

Lemma \ref{spl1} is a general method for introducing affine parameters
into real Hadamard matrices, irrespectively of their order, and it can
be applied to various other matrices as well. Unfortunately, however,
it cannot be applied to $\BH(4,n)$ matrices when $n\equiv 2$ (mod
$4$). In the following we present a method working for some of these
BH(4,n) matrices.

Let us denote by $1^m$ the all-$1$ row vector of length $m$, and let
$\left\langle .,.\right\rangle$ denote the standard inner product in
$\mathbb{C}^d$ with the convention that it is linear in the first, and
conjugate linear in its second argument.

\begin{theorem}\label{ch1newp}
Let $H$ be a dephased complex Hadamard matrix with the following block structure
\[H=\left[\begin{array}{ccccc}
1 & 1 & 1 & 1^p & 1^q\\
1 & a & b & x & y\\
1 & b & a & x & -y\\
(1^r)^T & z^T & z^T & A & B\\
(1^s)^T & w^T & -w^T & C & D\\
\end{array}\right],\]
where $a$ and $b$ are arbitrary unimodular numbers. Then, after
replacing the row vectors $y$ with $\alpha y$ and $w$ with
$\overline{\alpha} w$ we obtain a one-parameter family of complex
Hadamard matrices $H(\alpha)$ where $\alpha$ is unimodular. If, in
addition, $b=a$ we can continue by replacing $w$ in $H(\alpha)$
with $\alpha\beta w$ to obtain a two-parameter family of complex
Hadamard matrices $H(\alpha,\beta)$ where $\alpha$ and $\beta$ are
unimodular.
\end{theorem}

\begin{proof}
We need to show that the rows of $H(\alpha,\beta)$ are pairwise
orthogonal. From the orthogonality of the first three rows of $H$
we get
\begin{equation*}
  \left.
  \begin{array}{lcc}
    \left\langle [1, 1, 1, 1^p, 1^q], [1, a , b, x, y]\right\rangle &
    = & 0 \\ 
    \left\langle [1, 1, 1, 1^p, 1^q], [1, b , a, x,
      -y]\right\rangle & = & 0 \\
  \end{array}
  \right\} \implies \left\langle 1^q, y \right\rangle=0,
\end{equation*}
\noindent
and
\begin{equation*}
  \begin{array}{rcl}
    0 = \left\langle [1, a, b, x, y], [1, b , a, x, -y]\right\rangle &
    = & 1 + a \overline{b} + b \overline{a} + p - q\\ 
    &= & \left\langle [1, a, b, x, \alpha y], [1, a , b, x, -\alpha
      y]\right\rangle, \\
  \end{array}
\end{equation*}
\noindent
and hence the first three rows of $H(\alpha,\beta)$ are pairwise
orthogonal. Similarly, it is easily seen that the rest of the rows
(beyond the first three) are pairwise orthogonal within
themselves. Additionally, the first row is trivially orthogonal to all
further rows. Therefore it remains to be seen that the second and
third rows of $H(\alpha,\beta)$ are orthogonal to all rows below them.

We show first that they are orthogonal to the rows which are of type $[1,
  z_i, z_i, A_i, B_i]$, $i=1,\hdots, r$. In the original matrix $H$
(i.e.,\ prior to parametrizing) we have 
\beql{ch3pars1}
1+z_i(\overline{a}+\overline{b})+\left\langle
A_i,x\right\rangle+\left\langle B_i, y\right\rangle=0, \eeq
\beql{ch3pars2} 1+z_i(\overline{a}+\overline{b})+\left\langle
A_i,x\right\rangle-\left\langle B_i, y\right\rangle=0, \eeq 
and hence
$\left\langle B_i,y\right\rangle=0$ for every $i=1,\hdots,r$. It
follows that after parametrization equations
\eqref{ch3pars1} and \eqref{ch3pars2} remain valid.

We proceed by proving that rows that are of type $[1, w_i, -w_i, C_i,
  D_i]$, $i=1,\hdots, s$, after parametrization, are orthogonal to the
second and third row of $H(\alpha,\beta)$. Again, in the original
matrix $H$ we have 
\beql{ch3pars3}
1+w_i\overline{a}-w_i\overline{b}+\left\langle
C_i,x\right\rangle+\left\langle D_i, y\right\rangle=0, \eeq
\beql{ch3pars4} 1-w_i\overline{a}+w_i\overline{b}+\left\langle
C_i,x\right\rangle-\left\langle D_i, y\right\rangle=0, \eeq 
and hence $\left\langle C_i,x\right\rangle=-1$ for every $i=1,\hdots,
s$. It follows, that \eqref{ch3pars3} are \eqref{ch3pars4} are valid,
provided that
\beql{ch3pars5} w_i\overline{a}-w_i\overline{b}+\left\langle D_i,
y\right\rangle=0 \eeq 
holds and therefore \eqref{ch3pars3} and \eqref{ch3pars4} remains true, after
parametrization. If, in addition, $b=a$, then $\left\langle D_i,
y\right\rangle=0$ for every $i=1,\hdots, s$, and hence
\eqref{ch3pars5} holds, independently of the scalar factor in $w$.

The one-parameter family $H(\alpha)$ can be considered as
$H(\alpha,\overline{\alpha})$. The equations \eqref{ch3pars1} -
\eqref{ch3pars4} hold as the condition $a = b$ is not required for
them. From the original matrix $H$ we get
\begin{equation*}
  w_i\overline{a}-w_i\overline{b}+\left\langle D_i,
  y\right\rangle = 0 \implies 
  \overline{\alpha}w_i\overline{a}-\overline{\alpha}w_i\overline{b}+\left\langle
  D_i, \alpha y\right\rangle = 0,
\end{equation*}
and \eqref{ch3pars5} holds for $H(\alpha,\overline{\alpha})$ also when
$a \neq b$.
\end{proof}

\begin{rem}
If $a$ and $b$ are as in Theorem \ref{ch1newp}, then it is easy to see
that the real part of $a\overline{b}$ is an integral number, and
therefore $b\in \pm a\cdot\{1,\omega,\omega^2,\mathbf{i}\}$, where
$\omega$ is a primitive complex third root of unity. Therefore one
hopes to apply the parametrizing scheme described for complex Hadamard
matrices with fourth and/or sixth roots of unity.
\end{rem}

Theorem \ref{ch1newp} describes a local property of the complex
Hadamard matrix $H$. Its conditions can be fairly easily checked, even
by hand, and it can be implemented as a computer program to construct
infinite families automatically. 

It is natural to ask how many degrees of freedom (i.e.,\ independent,
smooth parameters) can be introduced into a given complex Hadamard
matrix. We recall the following fundamental concept from \cite{karol}:
The defect $\mathrm{d}(H)$ of an $n\times n$ complex Hadamard matrix
$H$ reads $\mathrm{d}(H)=m-2n+1$, where $m$ is the dimension of the
solution space of the following real linear system with respect to the
matrix variable $R\in\mathbb{R}^{n\times n}$: \beql{s1def}
\sum\limits_{k=1}^n
H_{i,k}\overline{H}_{j,k}\left(R_{i,k}-R_{j,k}\right)=0,\ \ \ \ 1\leq
i<j\leq n.  \eeq

In what follows we motivate the formula \eqref{s1def}. Consider a
dephased complex Hadamard matrix $H$ and a phasing matrix $R$ whose
first row and column is $0$, and all other entries are a real
variables. Now consider the matrix
$K:=H\circ\mathrm{EXP}\left(\mathbf{i}R\right)$, where $\circ$ is the
entry-wise product and $\mathrm{EXP}$ is the entry-wise exponential
function, respectively. The resulting matrix $K$ is the most general
parametrized matrix, stemming from $H$. Note that $K$ is unimodular,
but not necessarily complex Hadamard. In order to ensure this latter
condition we spell out the orthogonality conditions of $K$, obtaining
\beql{s1gen}\left\langle K_i,K_j\right\rangle=\sum_{k=1}^n
H_{ik}\overline{H}_{kj}\mathbf{e}^{\mathbf{i}(R_{ik}-R_{kj})}=0,\ \ \ \ 1\leq
i<j\leq n, \eeq which, after linearizing the exponential function
(i.e.,\ replacing it with its first order Taylor expansion) leads to
\eqref{s1def}. Therefore those phasing matrices $R$ satisfying the
real linear system \eqref{s1def} lead to parametric families of
complex Hadamard matrices in a neighborhood of the initial matrix $H$,
up to first order; however, \eqref{s1gen} is far more restrictive
further decreasing the number of parameters in $R$ in general.

Note that the degree of freedom $m$ in the defect of an $n \times n$
matrix is decreased by $2n-1$ as this many parameters can always be 
introduced into a complex Hadamard matrix via multiplication by
unitary diagonal matrices. This operation, however, does not yield new
complex Hadamard matrices, up to equivalence, and therefore only dephased
families are considered. Matrices that cannot be parametrized in
any other way are called isolated. The most important properties
of the defect are summarized in the following result from
\cite{karol,TZ2}.

\begin{proposition}\label{propdef}
Let $H$ be a complex Hadamard matrix. Then
\begin{enumerate}[$($a$)$]
\item $\mathrm{d}(H)$ is an invariant, up to equivalence, moreover
  $\mathrm{d}(H)=\mathrm{d}(H^\ast)=\mathrm{d}(\overline{H})=\mathrm{d}(H^T)$;
\item the number of smooth parameters which can be introduced into $H$
  is at most $\mathrm{d}(H)$;
\item if $\mathrm{d}(H)=0$ then $H$ is isolated amongst all $n\times n$ complex Hadamard matrices.
\end{enumerate}
\end{proposition}

Note that part (c) does not require the smoothness condition and as a
result it does not follow from part (b).

\begin{rem}\label{rem6}
It might be possible to use \eqref{s1def} to extend known affine
families with further parameters as follows. Suppose that
$H\left(\alpha\right)$ is a dephased parametric family of complex
Hadamard matrices. Let $k>1$ be an integer and evaluate
$H\left(\alpha\right)$ at $k$ random points to obtain a series of
dephased complex Hadamard matrices $H_1, H_2,\hdots, H_k$. Now
consider a phasing matrix $R$ (whose entries are real variables) and
the resulting general parametric matrices
$K_i=H_i\circ\mathrm{EXP}\left(\mathbf{i}\mathrm{R}\right)$ and the
corresponding $k$ system of linear equations given by \eqref{s1def}
for $i=1,\hdots, k$. A common solution of this $k$ system, which
should be further refined by \eqref{s1gen}, might result in families
of complex Hadamard matrices with more free parameters than the
initial one $H\left(\alpha\right)$.
\end{rem}

It is easy to see that a dephased complex Hadamard matrix is
equivalent to a $\BH(q,n)$ matrix if and only if all of its entries are
some $q$th root of unity; consequently it is easy to obtain all
$\BH(q,n)$ matrices which are members of some parametric family
$H(\alpha)$, stemming from the starting point $\BH(q,n)$ matrix $H$.

\section{The quaternary complex Hadamard matrices of order $10$}\label{ch4}

In this section we report on our findings regarding $\BH(4,10)$
matrices.

\begin{theorem}\label{prop1a}
There are exactly $10$ $\BH(4,10)$ matrices, up to equivalence, forming
$7$ $\mathrm{ACT}$-equivalence classes.
\end{theorem}

Now we present affine families containing all inequivalent $\BH(4,n)$
matrices. The first family, \[
D_{10}^{(3)}(a,b,c)=\left[
\begin{array}{rrrrrrrrrrrr}
 1 & 1 & 1 & 1 & 1 & 1 & 1 & 1 & 1 & 1 \\
 1 & -1 & -\mathbf{i} a\overline{b} & -\mathbf{i} a & -\mathbf{i}\overline{c} & -\mathbf{i} & \mathbf{i}\overline{c} & \mathbf{i} a & \mathbf{i} a\overline{b} &
   \mathbf{i} \\
 1 & -\mathbf{i} b\overline{a} & -1 & \mathbf{i} b & \mathbf{i}\overline{c} & -\mathbf{i} & -\mathbf{i}\overline{c} & -\mathbf{i} b & \mathbf{i} & \mathbf{i}
   b\overline{a} \\
 1 & -\mathbf{i}\overline{a} & \mathbf{i}\overline{b} & -1 & -\mathbf{i} & \mathbf{i} & -\mathbf{i} & \mathbf{i} & -\mathbf{i}\overline{b} & \mathbf{i}\overline{a} \\
 1 & -\mathbf{i} c & \mathbf{i} c & -\mathbf{i} & -1 & \mathbf{i} & \mathbf{i} & -\mathbf{i} & \mathbf{i} c & -\mathbf{i} c \\
 1 & -\mathbf{i} & -\mathbf{i} & \mathbf{i} & \mathbf{i} & -1 & \mathbf{i} & \mathbf{i} & -\mathbf{i} & -\mathbf{i} \\
 1 & \mathbf{i} c & -\mathbf{i} c & -\mathbf{i} & \mathbf{i} & \mathbf{i} & -1 & -\mathbf{i} & -\mathbf{i} c & \mathbf{i} c \\
 1 & \mathbf{i}\overline{a} & -\mathbf{i}\overline{b} & \mathbf{i} & -\mathbf{i} & \mathbf{i} & -\mathbf{i} & -1 & \mathbf{i}\overline{b} & -\mathbf{i}\overline{a} \\
 1 & \mathbf{i} b\overline{a} & \mathbf{i} & -\mathbf{i} b & \mathbf{i}\overline{c} & -\mathbf{i} & -\mathbf{i}\overline{c} & \mathbf{i} b & -1 & -\mathbf{i}
   b\overline{a} \\
 1 & \mathbf{i} & \mathbf{i} a\overline{b} & \mathbf{i} a & -\mathbf{i}\overline{c} & -\mathbf{i} & \mathbf{i}\overline{c} & -\mathbf{i} a & -\mathbf{i} a\overline{b} &
   -1
\end{array}
\right], \] has been obtained in \cite{SZF1}. It contains three out of
the seven $\mathrm{ACT}$-classes, namely $D_{10}^{(3)}(1,1,1)$,
$D_{10}^{(3)}(1,1,\mathbf{i})$ and
$D_{10}^{(3)}(1,\mathbf{i},\mathbf{i})$. Note that its orbit can be
constructed by repeatedly using Theorem \ref{ch1newp}.

The second family reads

\[N_{10B}^{(3)}(a,b,c)=\left[
\begin{array}{rrrrrrrrrrrr}
 1 & 1 & 1 & 1 & 1 & 1 & 1 & 1 & 1 & 1 \\
 1 & 1 & -1 & -1 & 1 & 1 & -\mathbf{i} & -1 & -1 & \mathbf{i} \\
 1 & a & c & -\mathbf{i} c & -a & -1 & -c & -\mathbf{i} a b c & \mathbf{i} a b c & \mathbf{i} c \\
 1 & -a & -\mathbf{i} c & c & a & -1 & -c & \mathbf{i} a b c & -\mathbf{i} a b c & \mathbf{i} c \\
 1 & -\mathbf{i} & -\mathbf{i} \overline{a}c & \mathbf{i} \overline{a}c & -1 & 1 & \mathbf{i} & b c & -b c & -1 \\
 1 & \mathbf{i} \overline{b} & -\mathbf{i} \overline{ab}c & \mathbf{i} \overline{ab}c & -\mathbf{i} \overline{b} & -1 & c & -c & \mathbf{i} c & -\mathbf{i} c \\
 1 & -\mathbf{i} \overline{b} & \mathbf{i} \overline{ab}c & -\mathbf{i} \overline{ab}c & \mathbf{i} \overline{b} & -1 & c & \mathbf{i} c & -c & -\mathbf{i} c \\
 1 & -1 & \mathbf{i} \overline{a}c & -\mathbf{i} \overline{a}c & -\mathbf{i} & 1 & \mathbf{i} & - bc &  bc & -1 \\
 1 & \mathbf{i} & -c & -c & \mathbf{i} & -\mathbf{i} & -1 & c & c & -\mathbf{i} \\
 1 & -1 & \mathbf{i} c & \mathbf{i} c & -1 & \mathbf{i} & -\mathbf{i} & -\mathbf{i} c & -\mathbf{i} c & 1
\end{array}
\right],\]

\noindent
whose one-parametric subfamily $N_{10B}^{(1)}(a)=N_{10B}^{(3)}(a,1,1)$
was reported in \cite{TZW}. This matrix contains two
$\mathrm{ACT}$-classes, namely $N_{10B}^{(3)}(1,1,1)$ and
$N_{10B}^{(3)}(\mathbf{i},1,1)$, so despite the two-degree extension,
which was obtained by a repeated application of Theorem \ref{ch1newp},
no previously unknown $\BH(4,10)$ matrix surfaced.

The two remaining matrices can be obtained from complex Golay
sequences \cite{CHK}, and they belong to the family

\[G_{10}^{(1)}(a)=\left[
\begin{array}{rrrrr|rrrrr}
 1 & a & a^2 & \mathbf{i} a^3 & -\mathbf{i} a^4 & 1 & \mathbf{i} a & -\mathbf{i} a^2 & -a^3 & \mathbf{i} a^4 \\
 -\mathbf{i} a^4 & 1 & a & a^2 & \mathbf{i} a^3 & \mathbf{i} a^4 & 1 & \mathbf{i} a & -\mathbf{i} a^2 & -a^3 \\
 \mathbf{i} a^3 & -\mathbf{i} a^4 & 1 & a & a^2 & -a^3 & \mathbf{i} a^4 & 1 & \mathbf{i} a & -\mathbf{i} a^2 \\
 a^2 & \mathbf{i} a^3 & -\mathbf{i} a^4 & 1 & a & -\mathbf{i} a^2 & -a^3 & \mathbf{i} a^4 & 1 & \mathbf{i} a \\
 a & a^2 & \mathbf{i} a^3 & -\mathbf{i} a^4 & 1 & \mathbf{i} a & -\mathbf{i} a^2 & -a^3 & \mathbf{i} a^4 & 1 \\
 \hline
 1 & -\mathbf{i} \overline{a}^4 & -\overline{a}^3 & \mathbf{i} \overline{a}^2 & -\mathbf{i} \overline{a} & -1 & -\mathbf{i} \overline{a}^4 & \mathbf{i} \overline{a}^3 & -\overline{a}^2 & -\overline{a} \\
 -\mathbf{i} \overline{a} & 1 & -\mathbf{i} \overline{a}^4 & -\overline{a}^3 & \mathbf{i} \overline{a}^2 & -\overline{a} & -1 & -\mathbf{i} \overline{a}^4 & \mathbf{i} \overline{a}^3 & -\overline{a}^2 \\
 \mathbf{i} \overline{a}^2 & -\mathbf{i} \overline{a} & 1 & -\mathbf{i} \overline{a}^4 & -\overline{a}^3 & -\overline{a}^2 & -\overline{a} & -1 & -\mathbf{i} \overline{a}^4 & \mathbf{i} \overline{a}^3 \\
 -\overline{a}^3 & \mathbf{i} \overline{a}^2 & -\mathbf{i} \overline{a} & 1 & -\mathbf{i} \overline{a}^4 & \mathbf{i} \overline{a}^3 & -\overline{a}^2 & -\overline{a} & -1 & -\mathbf{i} \overline{a}^4 \\
 -\mathbf{i} \overline{a}^4 & -\overline{a}^3 & \mathbf{i} \overline{a}^2 & -\mathbf{i} \overline{a} & 1 & -\mathbf{i} \overline{a}^4 & \mathbf{i} \overline{a}^3 & -\overline{a}^2 & -\overline{a} & -1
\end{array}
\right].\]

The matrices $G_{10}^{(1)}(1)$ and $G_{10}^{(1)}(-1)$ are inequivalent
from all previously considered examples. We note here that the family
$G_{10}^{(1)}$ implies the existence of an infinite family of triplets
of pairwise mutually unbiased bases in $\mathbb{C}^{10}$ by a
construction of Zauner \cite{Z}.

\begin{theorem}\label{prop1b}
Each member of the $\mathrm{ACT}$-equivalence classes of $\BH(4,10)$
matrices can be obtained from three partially overlapping infinite
affine parametric families of complex Hadamard matrices. These
matrices, up to $\mathrm{ACT}$-equivalence, can be recognized by the
number of $3\times 3$ vanishing minors they contain.
\end{theorem}

It turns out that all $\BH(4,10)$ matrices appear in the existing
literature, however, the parametric families $N_{10B}^{(3)}(a,b,c)$
and $G_{10}^{(1)}(a)$ are considered here for the first time. The
authors believe that these families contain new, previously unknown
$\BH(q,10)$ matrices for some $q>4$. Table \ref{T1} summarizes the
properties of $\BH(4,10)$ matrices for which the legend is as follows:
the column ``ACT'' describes if the matrix is equivalent to its
adjoint, conjugate or transpose, respectively, while the column
``HBS'' indicates if it is equivalent to a Hermitian matrix, if
contains a $n/2\times n/2$ sub-Hadamard matrix (i.e.,\ it comes from a
doubling construction, see \cite{popa}) or if is equivalent to a
symmetric matrix, respectively; the rest of the column headings speak
for themselves.

\begin{table}[h]
\centering 
\caption{Summary of the $\BH(4,10)$ matrices}\label{T1} 
\begin{tabular}{cclccccccc} 
\hline
ACT & Equiv. & Family, & ACT & HBS & Auto. & Defect & Orbit & $\mathbb{Z}_4$ & Invariant\\
class & classes & coordinates & & & order & & & rank & $(3\times 3)$\\
\hline
$1$ & $1$     & $D_{10}^{(3)}(1,1,\mathbf{i})$  & YYY & Y--Y & 64   & 10 & 3 & 9 & 2032\\
$2$ & $2, 5$  & $N_{10B}^{(3)}(\mathbf{i},1,1)$ & NNY & N--Y & 192  & 11 & 3 & 9 & 2496\\
$3$ & $3$     & $D_{10}^{(3)}(1,1,1)$           & YYY & Y--Y & 2880 & 16 & 3 & 9 & 3600\\
$4$ & $4$     & $D_{10}^{(3)}(1,\mathbf{i},\mathbf{i})$ & YYY & Y--Y & 32   & 12 & 3 & 9 & 2080\\
$5$ & $6,8$   & $N_{10B}^{(3)}(1,1,1)$          & NNY & N--Y & 64   & 7  & 3 & 9 & 1568\\
$6$ & $7$     & $G_{10}^{(1)}(1)$               & YYY & Y--Y & 20   & 8  & 1 & 9 & 1580\\
$7$ & $9, 10$ & $G_{10}^{(1)}(-1)$              & NYN & N--N & 80   & 8  & 1 & 9 & 1600\\
\hline
\end{tabular}
\end{table}

\section{The quaternary complex Hadamard matrices of order $12$}\label{ch5}

The main result concerning $\BH(4,12)$ matrices is the following.

\begin{theorem}
There are exactly $319$ $\BH(4,12)$ matrices, up to equivalence,
forming $167$ $\mathrm{ACT}$-equivalence classes.
\end{theorem}

Because all but four representatives of the $\mathrm{ACT}$-equivalence
classes contain a pair of real rows, almost all matrices belong to
some parametric families through Lemma \ref{spl1}. It turned out that
the exceptional four matrices can be parametrized as well by solving
\eqref{s1def} and \eqref{s1gen}, and they can be described by two
parametric families (consult the families $L_{12A}^{(1)}(a)$ and
$L_{12B}^{(1)}(a)$ in Appendix \ref{APPA}). Therefore every
$\BH(4,12)$ matrix belongs to some parametric family. We, however,
tried to minimize the number of parametric families required for the
presentation of all $\BH(4,12)$ matrices by attempting to describe
families capturing essentially different properties of these
matrices. We managed to describe the matrices with the aid of 23
families. In what follows we display four of the $23$ families
containing altogether a considerable number of distinct
$\mathrm{ACT}$-classes.

The first family, stemming from the real Hadamard matrix $H_{12}$, has
$10$ free parameters, and is the largest known affine family of order
$12$. It contains $36$ distinct $\mathrm{ACT}$-classes.

\[
H_{12B}^{(10)}(a,b,c,d,e,f,g,h,i,j)=
\]
\[\left[
\begin{array}{rrrrrrrrrrrr}
 1 & 1 & 1 & 1 & 1 & 1 & 1 & 1 & 1 & 1 & 1 & 1 \\
 1 & 1 & 1 & 1 & 1 & 1 & -1 & -1 & -1 & -1 & -1 & -1 \\
 1 & 1 & -1 & -1 & a & -a & c & -c & e & e & -e & -e \\
 1 & 1 & -1 & -1 & -a & a & c & -c & -e & -e & e & e \\
 1 & 1 & b & -b & -1 & -1 & -c & c & g & -g & h & -h \\
 1 & 1 & -b & b & -1 & -1 & -c & c & -g & g & -h & h \\
 1 & -1 & d & d & -d & -d & c d & -c d & f g\overline{b} & -f g\overline{b} & -f
   h\overline{b} & f h\overline{b} \\
 1 & -1 & -d & -d & d & d & -c d & c d & f g\overline{b} & -f g\overline{b} & -f
   h\overline{b} & f h\overline{b} \\
 1 & -1 & f & -f & i & -i & -e i\overline{a} & -e i\overline{a} & -f g\overline{b} & f
   g\overline{b} & e i\overline{a} & e i\overline{a} \\
 1 & -1 & f & -f & -i & i & e i\overline{a} & e i\overline{a} & -f g\overline{b} & f
   g\overline{b} & -e i\overline{a} & -e i\overline{a} \\
 1 & -1 & -f & f & j & -j & e j\overline{a} & e j\overline{a} & -e j\overline{a} & -e
   j\overline{a} & f h\overline{b} & -f h\overline{b} \\
 1 & -1 & -f & f & -j & j & -e j\overline{a} & -e j\overline{a} & e j\overline{a} & e
   j\overline{a} & f h\overline{b} & -f h\overline{b}
\end{array}
\right]
\]
The second matrix has $8$ parameters and contains $43$
distinct $\mathrm{ACT}$-equivalence classes of the $\BH(4,12)$
matrices.  \[ H_{12C}^{(8)}(a,b,c,d,e,f,g,h)= \]
\[\left[
\begin{array}{rrrrrrrrrrrr}
 1 & 1 & 1 & 1 & 1 & 1 & 1 & 1 & 1 & 1 & 1 & 1 \\
 1 & 1 & c & 1 & 1 & a & -a & -c & -1 & -1 & -1 & -1 \\
 1 & 1 & -1 & -b & b & -a & a & -1 & e & d & -d & -e \\
 1 & 1 & -f & -1 & -1 & f & f & -f & -e f & -d & d & e f \\
 1 & 1 & f & -1 & -1 & -f & -f & f & e f & -d & d & -e f \\
 1 & 1 & -c & b & -b & -1 & -1 & c & -e & d & -d & e \\
 1 & -1 & c g h & b g h & -b g h & -a g h & a g h & -c g h & h & -h & -h & h \\
 1 & -1 & -c g h & -b g h & b g h & a g h & -a g h & c g h & h & -h & -h & h \\
 1 & -1 & c h & -b h & b h & -h & -h & -c h & -e h & h & h & e h \\
 1 & -1 & h & -h & -h & h & h & h & -h & d h & -d h & -h \\
 1 & -1 & -c h & h & h & -a h & a h & c h & -h & -d h & d h & -h \\
 1 & -1 & -h & b h & -b h & a h & -a h & -h & e h & h & h & -e h
\end{array}
\right].\]

\noindent
The families $H_{12B}^{(10)}$ and $H_{12C}^{(8)}$ intersect in $19$
$\mathrm{ACT}$-classes and they represent $60$ distinct
$\mathrm{ACT}$-classes.  The next family comes from Di\c{t}\u{a}'s
general method \cite{dita3} and all members of it contain a
$6\times 6$ sub-Hadamard matrix $D_6^{(1)}(a)$, see [16]:
\[D_{12}^{(7)}(a,b,c,d,e,f,g)=\left[\begin{array}{rr}
D_6^{(1)}(a) & \mathrm{Diag}(1,b,c,d,e,f)D_6^{(1)}(g)\\
D_6^{(1)}(a) & -\mathrm{Diag}(1,b,c,d,e,f)D_6^{(1)}(g)\\
\end{array}\right].\]
The family $D_{12}^{(7)}$ contains $20$ $\mathrm{ACT}$-equivalence
classes, none of which is equivalent to the previously considered $60$
classes.

The fourth family, containing $22$ $\mathrm{ACT}$-classes, of which
$21$ is distinct from all previously discussed reads

\[X_{12}^{(7)}(a,b,c,d,e,f,g)=\]\[
\left[\begin{array}{rrrrrrrrrrrr}
 1 & 1 & 1 & 1 & 1 & 1 & 1 & 1 & 1 & 1 & 1 & 1 \\
 1 & 1 & 1 & 1 & \mathbf{i} & \mathbf{i} & -1 & -1 & -1 & -1 & -\mathbf{i} & -\mathbf{i} \\
 1 & 1 & \mathbf{i} & \mathbf{i} & -\mathbf{i} & -\mathbf{i} & \mathbf{i} a & \mathbf{i} b & -\mathbf{i} b & -\mathbf{i} a & -1 & -1 \\
 1 & 1 & -1 & -1 & e & -e & a & -\mathbf{i} b & \mathbf{i} b & -a & f & -f \\
 1 & 1 & -1 & -1 & -e & e & -\mathbf{i} a & b & -b & \mathbf{i} a & -f & f \\
 1 & 1 & -\mathbf{i} & -\mathbf{i} & -1 & -1 & -a & -b & b & a & \mathbf{i} & \mathbf{i} \\
 1 & -1 & c g & -c g & g & -g & -a g\overline{e} & b g\overline{e} & -b g\overline{e} & a g\overline{e} & -\mathbf{i} f g\overline{e} & \mathbf{i} f
   g\overline{e} \\
 1 & -1 & c g & -c g & \mathbf{i} g & -\mathbf{i} g & a g\overline{e} & -b g\overline{e} & b g\overline{e} & -a g\overline{e} & -f g\overline{e} & f
   g\overline{e} \\
 1 & -1 & \mathbf{i} c g & -\mathbf{i} c g & -\mathbf{i} g & \mathbf{i} g & d g & -d g & -d g & d g & \mathbf{i} f g\overline{e} & -\mathbf{i} f g\overline{e} \\
 1 & -1 & -c g & c g & \mathbf{i} d g & \mathbf{i} d g & -\mathbf{i} d g & d g & d g & -\mathbf{i} d g & -d g & -d g \\
 1 & -1 & -c g & c g & -\mathbf{i} d g & -\mathbf{i} d g & -d g & \mathbf{i} d g & \mathbf{i} d g & -d g & d g & d g \\
 1 & -1 & -\mathbf{i} c g & \mathbf{i} c g & -g & g & \mathbf{i} d g & -\mathbf{i} d g & -\mathbf{i} d g & \mathbf{i} d g & f g\overline{e} & -f g\overline{e}
\end{array}
\right]\] The four families $H_{12B}^{(10)}$, $H_{12C}^{(8)}$,
$D_{12}^{(7)}$ and $X_{12}^{(7)}$ contain $101$
$\mathrm{ACT}$-equivalence classes. If we include the adjoint,
conjugate and transpose of these matrices , we end up with $186$
matrices, up to equivalence. This is more than half of all
inequivalent $\BH(4,12)$ matrices. The further $19$ families containing
the remaining $\BH(4,12)$ matrices along with two summarizing tables
are available in Appendix \ref{APPA}.

\begin{theorem}
Each member of the $\mathrm{ACT}$-equivalence classes of $\BH(4,12)$
matrices can be obtained from $23$ partially overlapping infinite
affine parametric family of complex Hadamard matrices. These matrices,
up to $\mathrm{ACT}$-equivalence, can be recognized by the number of
$4\times 4$ and $5\times 5$ vanishing minors they contain.
\end{theorem}

\section{The quaternary complex Hadamard matrices of order $14$}\label{ch6}

The exhaustive computer search yielded the following result:

\begin{theorem}
There are precisely $752$ $\BH(4,14)$ matrices, up to equivalence,
forming $298$ $\mathrm{ACT}$-equivalence classes.
\end{theorem}

Further analysis revealed $8$ matrices that are isolated. For these
matrices there exists a neighbourhood of the matrix which does not
contain any inequivalent complex Hadamard matrices.  We display here
one of the $8$ known isolated $\BH(4,14)$ matrices:

\[L_{14A}^{(0)}=\left[
\begin{array}{rrrrrrrrrrrrrr}
 1 & 1 & 1 & 1 & 1 & 1 & 1 & 1 & 1 & 1 & 1 & 1 & 1 & 1 \\
 1 & 1 & 1 & -1 & -1 & 1 & -\mathbf{i} & -1 & -1 & \mathbf{i} & -\mathbf{i} & -1 & \mathbf{i} & 1 \\
 1 & 1 & \mathbf{i} & \mathbf{i} & \mathbf{i} & -\mathbf{i} & 1 & -\mathbf{i} & -1 & -\mathbf{i} & -1 & 1 & -1 & -1 \\
 1 & 1 & \mathbf{i} & -\mathbf{i} & -\mathbf{i} & -\mathbf{i} & -1 & \mathbf{i} & \mathbf{i} & -1 & 1 & \mathbf{i} & -\mathbf{i} & -1 \\
 1 & 1 & -\mathbf{i} & 1 & -1 & \mathbf{i} & -1 & -\mathbf{i} & 1 & 1 & \mathbf{i} & -1 & -1 & -1 \\
 1 & \mathbf{i} & -1 & -\mathbf{i} & -1 & -1 & -\mathbf{i} & \mathbf{i} & 1 & -\mathbf{i} & -1 & 1 & \mathbf{i} & 1 \\
 1 & \mathbf{i} & -\mathbf{i} & \mathbf{i} & 1 & -1 & -1 & -\mathbf{i} & -\mathbf{i} & -1 & -\mathbf{i} & \mathbf{i} & 1 & \mathbf{i} \\
 1 & -1 & 1 & 1 & -1 & \mathbf{i} & \mathbf{i} & \mathbf{i} & -1 & -\mathbf{i} & -\mathbf{i} & -\mathbf{i} & 1 & -1 \\
 1 & -1 & 1 & -\mathbf{i} & \mathbf{i} & -1 & 1 & -1 & -\mathbf{i} & \mathbf{i} & \mathbf{i} & \mathbf{i} & -\mathbf{i} & -\mathbf{i} \\
 1 & -1 & \mathbf{i} & -1 & 1 & 1 & -1 & -1 & 1 & 1 & -1 & -\mathbf{i} & -\mathbf{i} & \mathbf{i} \\
 1 & -1 & -1 & 1 & \mathbf{i} & -\mathbf{i} & -1 & 1 & -1 & \mathbf{i} & 1 & -\mathbf{i} & -1 & 1 \\
 1 & -1 & -\mathbf{i} & -1 & -\mathbf{i} & 1 & \mathbf{i} & -\mathbf{i} & \mathbf{i} & -1 & \mathbf{i} & 1 & \mathbf{i} & -\mathbf{i} \\
 1 & -\mathbf{i} & -1 & \mathbf{i} & 1 & \mathbf{i} & 1 & \mathbf{i} & 1 & -1 & -\mathbf{i} & -1 & -1 & -\mathbf{i} \\
 1 & -\mathbf{i} & -1 & -1 & -\mathbf{i} & -1 & 1 & 1 & -1 & 1 & \mathbf{i} & -1 & 1 & \mathbf{i}
\end{array}
\right].\] The defect $\mathrm{d}\left(L_{14A}^{(0)}\right)=0$ and
hence, by Proposition \ref{propdef}, the matrix $L_{14A}^{(0)}$ is
isolated amongst all $14\times 14$ complex Hadamard matrices. Three
additional inequivalent isolated matrices can be obtained by
considering the adjoint, conjugate and transpose of $L_{14A}^{(0)}$.

\begin{theorem}
There are at least $8$ isolated $\BH(4,14)$ matrices, up to equivalence, forming
two $\mathrm{ACT}$-equivalence classes. 
\end{theorem}

As we have found isolated matrices it follows that it is not possible
to come up with a universal parametrization scheme for the $\BH(4,n)$
matrices. This stands in contrast to the real Hadamard matrices which
can be parametrized always (cf. Lemma \ref{spl1}).

\begin{theorem}
All $\BH(4,14)$ matrices, up to $\mathrm{ACT}$-equivalence, can be recognized by
the number of $4\times 4$ vanishing minors they contain.
\end{theorem}

Although by using Theorem \ref{ch1newp} we were able to obtain various
parametric families starting from many of the $\BH(4,14)$ matrices we
constructed, we could not introduce affine parameters into matrices
having relatively small defect. Whether or not they are isolated
remains an open problem.

For an illustration of the parametrization of $\BH(4,14)$ matrices, we
display here a six-parameter family of complex Hadamard matrices:
\[D_{14}^{(6)}(a,b,c,d,e,f)=\]
\[\left[
\begin{array}{rrrrrrrrrrrrrr}
 1 & 1 & 1 & 1 & 1 & 1 & 1 & 1 & 1 & 1 & 1 & 1 & 1 & 1 \\
 1 & -1 & \mathbf{i} a & -\mathbf{i} b & \mathbf{i} f & \mathbf{i} c & -\mathbf{i} f & -\mathbf{i} & -\mathbf{i} f & -\mathbf{i} c & \mathbf{i} f & \mathbf{i} b & -\mathbf{i} a & \mathbf{i} \\
 1 & \mathbf{i}\overline{a} & -1 & \mathbf{i} f & -\mathbf{i} f & \mathbf{i} c\overline{a} & \mathbf{i} e\overline{a} & -\mathbf{i} & -\mathbf{i} e\overline{a} & -\mathbf{i} c\overline{a} & -\mathbf{i} f & \mathbf{i} f & \mathbf{i} & -\mathbf{i}\overline{a} \\
 1 & -\mathbf{i}\overline{b} & \mathbf{i}\overline{f} & -1 & \mathbf{i} d\overline{b} & -\mathbf{i}\overline{f} & \mathbf{i} e\overline{b} & \mathbf{i} & -\mathbf{i} e\overline{b} & -\mathbf{i}\overline{f} & -\mathbf{i} d\overline{b} & -\mathbf{i} & \mathbf{i}\overline{f} & \mathbf{i}\overline{b} \\
 1 & \mathbf{i}\overline{f} & -\mathbf{i}\overline{f} & \mathbf{i} b\overline{d} & -1 & \mathbf{i} c\overline{d} & -\mathbf{i} e\overline{d} & \mathbf{i} & \mathbf{i} e\overline{d} & -\mathbf{i} c\overline{d} & -\mathbf{i} & -\mathbf{i} b\overline{d} & -\mathbf{i}\overline{f} & \mathbf{i}\overline{f} \\
 1 & \mathbf{i}\overline{c} & \mathbf{i} a\overline{c} & -\mathbf{i} f & \mathbf{i} d\overline{c} & -1 & \mathbf{i} f & -\mathbf{i} & \mathbf{i} f & \mathbf{i} & -\mathbf{i} d\overline{c} & -\mathbf{i} f & -\mathbf{i} a\overline{c} & -\mathbf{i}\overline{c} \\
 1 & -\mathbf{i}\overline{f} & \mathbf{i} a\overline{e} & \mathbf{i} b\overline{e} & -\mathbf{i} d\overline{e} & \mathbf{i}\overline{f} & -1 & \mathbf{i} & -\mathbf{i} & \mathbf{i}\overline{f} & \mathbf{i} d\overline{e} & -\mathbf{i} b\overline{e} & -\mathbf{i} a\overline{e} & -\mathbf{i}\overline{f} \\
 1 & -\mathbf{i} & -\mathbf{i} & \mathbf{i} & \mathbf{i} & -\mathbf{i} & \mathbf{i} & -1 & \mathbf{i} & -\mathbf{i} & \mathbf{i} & \mathbf{i} & -\mathbf{i} & -\mathbf{i} \\
 1 & -\mathbf{i}\overline{f} & -\mathbf{i} a\overline{e} & -\mathbf{i} b\overline{e} & \mathbf{i} d\overline{e} & \mathbf{i}\overline{f} & -\mathbf{i} & \mathbf{i} & -1 & \mathbf{i}\overline{f} & -\mathbf{i} d\overline{e} & \mathbf{i} b\overline{e} & \mathbf{i} a\overline{e} & -\mathbf{i}\overline{f} \\
 1 & -\mathbf{i}\overline{c} & -\mathbf{i} a\overline{c} & -\mathbf{i} f & -\mathbf{i} d\overline{c} & \mathbf{i} & \mathbf{i} f & -\mathbf{i} & \mathbf{i} f & -1 & \mathbf{i} d\overline{c} & -\mathbf{i} f & \mathbf{i} a\overline{c} & \mathbf{i}\overline{c} \\
 1 & \mathbf{i}\overline{f} & -\mathbf{i}\overline{f} & -\mathbf{i} b\overline{d} & -\mathbf{i} & -\mathbf{i} c\overline{d} & \mathbf{i} e\overline{d} & \mathbf{i} & -\mathbf{i} e\overline{d} & \mathbf{i} c\overline{d} & -1 & \mathbf{i} b\overline{d} & -\mathbf{i}\overline{f} & \mathbf{i}\overline{f} \\
 1 & \mathbf{i}\overline{b} & \mathbf{i}\overline{f} & -\mathbf{i} & -\mathbf{i} d\overline{b} & -\mathbf{i}\overline{f} & -\mathbf{i} e\overline{b} & \mathbf{i} & \mathbf{i} e\overline{b} & -\mathbf{i}\overline{f} & \mathbf{i} d\overline{b} & -1 & \mathbf{i}\overline{f} & -\mathbf{i}\overline{b} \\
 1 & -\mathbf{i}\overline{a} & \mathbf{i} & \mathbf{i} f & -\mathbf{i} f & -\mathbf{i} c\overline{a} & -\mathbf{i} e\overline{a} & -\mathbf{i} & \mathbf{i} e\overline{a} & \mathbf{i} c\overline{a} & -\mathbf{i} f & \mathbf{i} f & -1 & \mathbf{i}\overline{a} \\
 1 & \mathbf{i} & -\mathbf{i} a & \mathbf{i} b & \mathbf{i} f & -\mathbf{i} c & -\mathbf{i} f & -\mathbf{i} & -\mathbf{i} f & \mathbf{i} c & \mathbf{i} f & -\mathbf{i} b & \mathbf{i} a & -1
\end{array}
\right].\] The starting-point matrix $D_{14}^{(6)}(1,1,1,1,1,1)$ is a
symmetric conference matrix, which was considered along with the
five-parameter subfamily
$D_{14}^{(5)}(a,b,c,d,e)=D_{14}^{(6)}(a,b,c,d,e,1)$ in
\cite{SZF1}. The extra parameter was found by the argument outlined in
Remark \ref{rem6}: we have evaluated the known five-parameter family
$D_{14}^{(6)}(a,b,c,d,e,1)$ at random fourth roots of unity to obtain
some $\BH(4,14)$ matrices, say $H_1, H_2,\hdots, H_{10}$, and
considered the corresponding $10$ instances of \eqref{s1def} featuring
the same phasing matrix $R$ in all ten cases. These extra
equations heavily constrained the matrix $R$ resulting in the
one-parametric extension we discovered. Despite its large degree of
freedom the family $D_{14}^{(6)}$ contains $14$ $\mathrm{ACT}$-classes
only, and it seems that a given affine parametric family cannot
contain significantly more different $\mathrm{ACT}$-classes in order
$14$. 

For a summarizing table highlighting the main features of $\BH(4,14)$
matrices and for additional examples of parametric families consult a
Web site \cite{web}. All results of this work are available in full
detail at the Web site which describes all inequivalent $\BH(4,n)$
matrices, parametric families and summarizing tables for $10\leq n\leq
14$.

\appendix

\section{A list of parametric families of $\BH(4,12)$ matrices}\label{APPA}

In the following we display the $19$ parametric families of order $12$
which were not shown in Section \ref{ch5}. These $19$ families with
the families $H_{12B}^{(10)}$, $H_{12C}^{(8)}$, $D_{12}^{(7)}$,
$X_{12}^{(7)}$ (defined in Section \ref{ch5}) account for all $\BH(4,12)$
matrices, up to $\mathrm{ACT}$-equivalence: \tiny
\[L_{12A}^{(1)}(a)=\left[
\begin{array}{rrrrrrrrrrrr}
 1 & 1 & 1 & 1 & 1 & 1 & 1 & 1 & 1 & 1 & 1 & 1 \\
 1 & 1 & 1 & 1 & \mathbf{i} & \mathbf{i} & -1 & -1 & -1 & -1 & -\mathbf{i} & -\mathbf{i} \\
 1 & 1 & \overline{a} & -\overline{a} & -1 & -\mathbf{i} & \mathbf{i} a & \mathbf{i} & -\mathbf{i} & -\mathbf{i} a & \mathbf{i} & -1 \\
 1 & 1 & -\overline{a} & \overline{a} & -\mathbf{i} & -1 & \mathbf{i} a & -\mathbf{i} & \mathbf{i} & -\mathbf{i} a & -1 & \mathbf{i} \\
 1 & \mathbf{i} & -1 & -\mathbf{i} & a & -\mathbf{i} a & -\mathbf{i} a & \mathbf{i} a & -a & -a & \mathbf{i} a & a \\
 1 & \mathbf{i} & -\mathbf{i} & -1 & \mathbf{i} a & -a & a & -\mathbf{i} a & a & \mathbf{i} a & -a & -\mathbf{i} a \\
 1 & -1 & a & \mathbf{i} a & -\mathbf{i} a & a^2 & -\mathbf{i} a^2 & -\mathbf{i} a^2 & \mathbf{i} a^2 & -a^2 & \mathbf{i} a^2 & -a \\
 1 & -1 & \mathbf{i} a & a & -\mathbf{i} a & -a^2 & a^2 & \mathbf{i} a^2 & -\mathbf{i} a^2 & \mathbf{i} a^2 & -\mathbf{i} a^2 & -a \\
 1 & -1 & -a & -\mathbf{i} a & \mathbf{i} a^2 & \mathbf{i} a & -a^2 & a^2 & a^2 & -\mathbf{i} a^2 & a & -a^2 \\
 1 & -1 & -\mathbf{i} a & -a & -\mathbf{i} a^2 & \mathbf{i} a & \mathbf{i} a^2 & -a^2 & -a^2 & a^2 & a & a^2 \\
 1 & -\mathbf{i} & \mathbf{i} & -1 & \mathbf{i} a & -\mathbf{i} a & -\mathbf{i} a & -a & \mathbf{i} a & a & -\mathbf{i} a & \mathbf{i} a \\
 1 & -\mathbf{i} & -1 & \mathbf{i} & -a & a & -a & a & -\mathbf{i} a & \mathbf{i} a & -a & a
\end{array}
\right],\]
\[L_{12B}^{(1)}(a)=\left[
\begin{array}{rrrrrrrrrrrr}
 1 & 1 & 1 & 1 & 1 & 1 & 1 & 1 & 1 & 1 & 1 & 1 \\
 1 & 1 & 1 & a & \mathbf{i} a & \mathbf{i} & -1 & -1 & -a & -1 & -\mathbf{i} & -\mathbf{i} a \\
 1 & 1 & 1 & -1 & -1 & -\mathbf{i} & 1 & \mathbf{i} & -1 & -\mathbf{i} & \mathbf{i} & -1 \\
 1 & 1 & -1 & \mathbf{i} a & a & -\mathbf{i} a & -\mathbf{i} & -a & -\mathbf{i} a & \mathbf{i} & -1 & \mathbf{i} a \\
 1 & \mathbf{i} & -1 & a & -a & -\mathbf{i} & -1 & 1 & \mathbf{i} a & a & -\mathbf{i} a & -a \\
 1 & \mathbf{i} & -\mathbf{i} & -\mathbf{i} a & -\mathbf{i} a & \mathbf{i} a & -1 & a & -\mathbf{i} a & -a & \mathbf{i} a & \mathbf{i} a \\
 1 & -1 & a & \mathbf{i} a^2 & -\mathbf{i} a^2 & -\mathbf{i} a & \mathbf{i} a & -\mathbf{i} a & \mathbf{i} a & -a & -a & a \\
 1 & -1 & \mathbf{i} & -\mathbf{i} a & a & -1 & 1 & -\mathbf{i} & -a & \mathbf{i} a & a & -a \\
 1 & -1 & -a & -\mathbf{i} a^2 & \mathbf{i} a^2 & a & -\mathbf{i} a & -a & \mathbf{i} a & -\mathbf{i} a & \mathbf{i} a & a \\
 1 & -1 & -\mathbf{i} & -a & \mathbf{i} a & -a & \mathbf{i} & \mathbf{i} a & -\mathbf{i} a & a & -\mathbf{i} a & a \\
 1 & -\mathbf{i} & \mathbf{i} & -a & -a & \mathbf{i} a & -\mathbf{i} & a & a & \mathbf{i} & -1 & -\mathbf{i} a \\
 1 & -\mathbf{i} & -1 & \mathbf{i} a & -\mathbf{i} a & \mathbf{i} & \mathbf{i} & -1 & a & -\mathbf{i} & 1 & -a
\end{array}
\right],\]
\[B_{12A}^{(4)}(a,b,c,d)=\left[
\begin{array}{rrrrrrrrrrrr}
 1 & 1 & 1 & 1 & 1 & 1 & 1 & 1 & 1 & 1 & 1 & 1 \\
 1 & 1 & 1 & 1 & \mathbf{i} & \mathbf{i} & -1 & -1 & -1 & -1 & -\mathbf{i} & -\mathbf{i} \\
 1 & 1 & a\overline{b} & -a\overline{b} & -1 & -\mathbf{i} & c\overline{b} & \mathbf{i} c\overline{b} & -c\overline{b} & -\mathbf{i} c\overline{b} & \mathbf{i} & -1 \\
 1 & 1 & -a\overline{b} & a\overline{b} & -1 & -\mathbf{i} & -c\overline{b} & -\mathbf{i} c\overline{b} & \mathbf{i} c\overline{b} & c\overline{b} & -1 & \mathbf{i} \\
 1 & 1 & -1 & -1 & 1 & -1 & -\mathbf{i} c^2\overline{b d} & \mathbf{i} c^2\overline{b d} & -\mathbf{i} c\overline{b} & \mathbf{i} c\overline{b} & -c\overline{d} & c\overline{d} \\
 1 & 1 & -1 & -1 & -\mathbf{i} & \mathbf{i} & \mathbf{i} c^2\overline{b d} & -\mathbf{i} c^2\overline{b d} & c\overline{b} & -c\overline{b} & c\overline{d} & -c\overline{d} \\
 1 & -1 & \mathbf{i} b & \mathbf{i} b & b & -\mathbf{i} b & -c & \mathbf{i} c & c & -\mathbf{i} c & -\mathbf{i} b & -b \\
 1 & -1 & \mathbf{i} a & -\mathbf{i} a & c & c & -\mathbf{i} c & -\mathbf{i} c & -c & -c & \mathbf{i} c & \mathbf{i} c \\
 1 & -1 & \mathbf{i} a & -\mathbf{i} a & -c & -c & c & c & \mathbf{i} c & \mathbf{i} c & -\mathbf{i} c & -\mathbf{i} c \\
 1 & -1 & -\mathbf{i} a & \mathbf{i} a & \mathbf{i} b d\overline{c} & -\mathbf{i} b d\overline{c} & \mathbf{i} c & -\mathbf{i} c & -\mathbf{i} d & \mathbf{i} d & \mathbf{i} b & -\mathbf{i} b \\
 1 & -1 & -\mathbf{i} a & \mathbf{i} a & -\mathbf{i} b d\overline{c} & \mathbf{i} b d\overline{c} & c & -c & \mathbf{i} d & -\mathbf{i} d & -b & b \\
 1 & -1 & -\mathbf{i} b & -\mathbf{i} b & -b & \mathbf{i} b & -c & \mathbf{i} c & -\mathbf{i} c & c & b & \mathbf{i} b
\end{array}
\right],\]
\[B_{12B}^{(5)}(a,b,c,d,e)=\left[
\begin{array}{rrrrrrrrrrrr}
 1 & 1 & 1 & 1 & 1 & 1 & 1 & 1 & 1 & 1 & 1 & 1 \\
 1 & 1 & 1 & 1 & \mathbf{i} & \mathbf{i} & -1 & -1 & -1 & -1 & -\mathbf{i} & -\mathbf{i} \\
 1 & 1 & \mathbf{i} & \mathbf{i} & -\mathbf{i} & -\mathbf{i} & b & \mathbf{i} a & -b & -\mathbf{i} a & -1 & -1 \\
 1 & 1 & -1 & -1 & a & \mathbf{i} b & -\mathbf{i} b & -\mathbf{i} a & b & -a & \mathbf{i} a & -b \\
 1 & 1 & -1 & -1 & -a & -\mathbf{i} b & -b & a & \mathbf{i} b & \mathbf{i} a & -\mathbf{i} a & b \\
 1 & 1 & -\mathbf{i} & -\mathbf{i} & -1 & -1 & \mathbf{i} b & -a & -\mathbf{i} b & a & \mathbf{i} & \mathbf{i} \\
 1 & -1 & e & -e & c & -c & -\mathbf{i} c & \mathbf{i} c & -\mathbf{i} c & \mathbf{i} c & -c & c \\
 1 & -1 & e & -e & -\mathbf{i} c & \mathbf{i} c & \mathbf{i} c & -\mathbf{i} c & \mathbf{i} c & -\mathbf{i} c & -\mathbf{i} c & \mathbf{i} c \\
 1 & -1 & \mathbf{i} e & -\mathbf{i} e & -c & c & \mathbf{i} d & a d\overline{b} & -\mathbf{i} d & -a d\overline{b} & \mathbf{i} c & -\mathbf{i} c \\
 1 & -1 & -e & e & \mathbf{i} a d\overline{b} & d & -\mathbf{i} d & -\mathbf{i} a d\overline{b} & -d & a d\overline{b} & -a d\overline{b} & \mathbf{i} d \\
 1 & -1 & -e & e & -\mathbf{i} a d\overline{b} & -d & d & -a d\overline{b} & \mathbf{i} d & \mathbf{i} a d\overline{b} & a d\overline{b} & -\mathbf{i} d \\
 1 & -1 & -\mathbf{i} e & \mathbf{i} e & \mathbf{i} c & -\mathbf{i} c & -d & \mathbf{i} a d\overline{b} & d & -\mathbf{i} a d\overline{b} & c & -c
\end{array}
\right],\]
\[B_{12C}^{(5)}(a,b,c,d,e)=\left[
\begin{array}{rrrrrrrrrrrr}
 1 & 1 & 1 & 1 & 1 & 1 & 1 & 1 & 1 & 1 & 1 & 1 \\
 1 & 1 & 1 & 1 & \mathbf{i} & \mathbf{i} & -1 & -1 & -1 & -1 & -\mathbf{i} & -\mathbf{i} \\
 1 & 1 & a d\overline{b} & -a d\overline{b} & -1 & -\mathbf{i} & a & \mathbf{i} a & -a & -\mathbf{i} a & \mathbf{i} & -1 \\
 1 & 1 & -a d\overline{b} & a d\overline{b} & -1 & -\mathbf{i} & -a & -\mathbf{i} a & a & \mathbf{i} a & \mathbf{i} & -1 \\
 1 & 1 & -1 & -1 & 1 & -1 & \mathbf{i} b\overline{c} & -\mathbf{i} b\overline{c} & \mathbf{i} b\overline{c} & -\mathbf{i} b\overline{c} & -1 & 1 \\
 1 & 1 & -1 & -1 & -\mathbf{i} & \mathbf{i} & -\mathbf{i} b\overline{c} & \mathbf{i} b\overline{c} & -\mathbf{i} b\overline{c} & \mathbf{i} b\overline{c} & -\mathbf{i} & \mathbf{i} \\
 1 & -1 & e & e & -\mathbf{i} e & -e & -\mathbf{i} b & -b & \mathbf{i} b & b & -e & \mathbf{i} e \\
 1 & -1 & d & -d & b & b & -b & -\mathbf{i} b & -\mathbf{i} b & -b & \mathbf{i} b & \mathbf{i} b \\
 1 & -1 & d & -d & -b & -b & \mathbf{i} b & b & b & \mathbf{i} b & -\mathbf{i} b & -\mathbf{i} b \\
 1 & -1 & -d & d & \mathbf{i} c & -\mathbf{i} c & b & \mathbf{i} b & -\mathbf{i} b & -b & -c & c \\
 1 & -1 & -d & d & -\mathbf{i} c & \mathbf{i} c & \mathbf{i} b & b & -b & -\mathbf{i} b & c & -c \\
 1 & -1 & -e & -e & \mathbf{i} e & e & -\mathbf{i} b & -b & \mathbf{i} b & b & e & -\mathbf{i} e
\end{array}
\right],\]
\[B_{12D}^{(4)}(a,b,c,d)=\left[
\begin{array}{rrrrrrrrrrrr}
 1 & 1 & 1 & 1 & 1 & 1 & 1 & 1 & 1 & 1 & 1 & 1 \\
 1 & 1 & 1 & 1 & 1 & \mathbf{i} d & -1 & -1 & -1 & -1 & -1 & -\mathbf{i} d \\
 1 & 1 & 1 & -1 & -1 & -1 & 1 & \mathbf{i} c\overline{b} & \mathbf{i} c\overline{b} & -\mathbf{i} c\overline{b} & -\mathbf{i} c\overline{b} & -1 \\
 1 & 1 & -1 & 1 & -1 & -\mathbf{i} d & -1 & c^2\overline{b^2} & -c^2\overline{b^2} & c^2\overline{b^2} & -c^2\overline{b^2} & \mathbf{i} d \\
 1 & 1 & -1 & -1 & \mathbf{i} c\overline{b} & c d\overline{b} & \mathbf{i} c\overline{b} & -\mathbf{i} c\overline{b} & -\mathbf{i} c\overline{b} & -c^2\overline{b^2} & c^2\overline{b^2} & -c d\overline{b} \\
 1 & 1 & -1 & -1 & -\mathbf{i} c\overline{b} & -c d\overline{b} & -\mathbf{i} c\overline{b} & -c^2\overline{b^2} & c^2\overline{b^2} & \mathbf{i} c\overline{b} & \mathbf{i} c\overline{b} & c d\overline{b} \\
 1 & -1 & a b\overline{c} & a b\overline{c} & -a b\overline{c} & -\mathbf{i} a b d\overline{c} & -a b\overline{c} & -a c\overline{b} & a c\overline{b} & -a c\overline{b} & a c\overline{b} & \mathbf{i} a b d\overline{c} \\
 1 & -1 & a b\overline{c} & -a b\overline{c} & \mathbf{i} a & -\mathbf{i} a & -\mathbf{i} a & a c\overline{b} & -a c\overline{b} & \mathbf{i} a & \mathbf{i} a & -\mathbf{i} a \\
\end{array}
\right],\]
\[B_{12E}^{(5)}(a,b,c,d,e)=\left[
\begin{array}{rrrrrrrrrrrr}
 1 & 1 & 1 & 1 & 1 & 1 & 1 & 1 & 1 & 1 & 1 & 1 \\
 1 & a c\overline{b} & 1 & a c\overline{b} & \mathbf{i} b\overline{c} & \mathbf{i} a & -1 & -1 & -a c\overline{b} & -a c\overline{b} & -\mathbf{i} a & -\mathbf{i} b\overline{c} \\
 1 & 1 & 1 & -1 & -1 & -1 & 1 & 1 & \mathbf{i} b e & -\mathbf{i} b e & -1 & -1 \\
 1 & c d\overline{b} & -1 & c d\overline{b} & -c d\overline{b} & -\mathbf{i} d & \mathbf{i} c^2 d\overline{b^2} & -\mathbf{i} c^2 d\overline{b^2} & -\mathbf{i} c d e & \mathbf{i} c d e & \mathbf{i} d & -c d\overline{b} \\
 1 & c d\overline{b} & -1 & -c d\overline{b} & c d\overline{b} & -\mathbf{i} d & -\mathbf{i} c^2 d\overline{b^2} & \mathbf{i} c^2 d\overline{b^2} & -c d\overline{b} & -c d\overline{b} & \mathbf{i} d & c d\overline{b} \\
 1 & c d\overline{b} & -1 & -c d\overline{b} & -\mathbf{i} d & \mathbf{i} d & -c d\overline{b} & -c d\overline{b} & c d\overline{b} & c d\overline{b} & -\mathbf{i} d & \mathbf{i} d \\
 1 & -1 & 1 & 1 & -\mathbf{i} b\overline{c} & -1 & -\mathbf{i} c\overline{b} & \mathbf{i} c\overline{b} & -\mathbf{i} b e & \mathbf{i} b e & -1 & \mathbf{i} b\overline{c} \\
 1 & -a c\overline{b} & 1 & -a c\overline{b} & \mathbf{i} b\overline{c} & -\mathbf{i} a & -1 & -1 & a c\overline{b} & a c\overline{b} & \mathbf{i} a & -\mathbf{i} b\overline{c} \\
\end{array}
\right],\]
\[B_{12F}^{(7)}(a,b,c,d,e,f,g)=\left[
\begin{array}{rrrrrrrrrrrrr}
 1 & 1 & 1 & 1 & 1 & 1 & 1 & 1 & 1 & 1 & 1 & 1 \\
 1 & 1 & 1 & 1 & f\overline{b} & \mathbf{i} g\overline{b} & -1 & -f\overline{b} & -1 & -1 & -1 & -\mathbf{i} g\overline{b} \\
 1 & 1 & a\overline{b} & -a\overline{b} & -f\overline{b} & -1 & d\overline{b} & f\overline{b} & \mathbf{i} e\overline{b} & -d\overline{b} & -\mathbf{i} e\overline{b} & -1 \\
 1 & 1 & -a\overline{b} & a\overline{b} & -1 & -\mathbf{i} g\overline{b} & \mathbf{i} e\overline{c} & -1 & -\mathbf{i} e\overline{c} & \mathbf{i} e\overline{c} & -\mathbf{i} e\overline{c} & \mathbf{i} g\overline{b} \\
 1 & 1 & -1 & -1 & \mathbf{i} e\overline{c} & e g\overline{b c} & -\mathbf{i} e\overline{c} & \mathbf{i} e\overline{c} & -\mathbf{i} e\overline{b} & -\mathbf{i} e\overline{c} & \mathbf{i} e\overline{b} & -e g\overline{b c} \\
 1 & 1 & -1 & -1 & -\mathbf{i} e\overline{c} & -e g\overline{b c} & -d\overline{b} & -\mathbf{i} e\overline{c} & \mathbf{i} e\overline{c} & d\overline{b} & \mathbf{i} e\overline{c} & e g\overline{b c} \\
 1 & -1 & \mathbf{i} b & \mathbf{i} b & -\mathbf{i} b & g & b e\overline{c} & -\mathbf{i} b & -b e\overline{c} & b e\overline{c} & -b e\overline{c} & -g \\
 1 & -1 & \mathbf{i} a & -\mathbf{i} a & e f\overline{c} & -b e\overline{c} & -\mathbf{i} d & -e f\overline{c} & b e\overline{c} & \mathbf{i} d & b e\overline{c} & -b e\overline{c} \\
\end{array}
\right],\]
\[B_{12G}^{(7)}(a,b,c,d,e,f,g)=\left[
\begin{array}{rrrrrrrrrrrr}
 1 & 1 & 1 & 1 & 1 & 1 & 1 & 1 & 1 & 1 & 1 & 1 \\
 1 & 1 & 1 & 1 & c\overline{b} & \mathbf{i} e\overline{b} & -c\overline{b} & -1 & -1 & -1 & -1 & -\mathbf{i} e\overline{b} \\
 1 & 1 & a\overline{b} & -a\overline{b} & -c\overline{b} & -1 & c\overline{b} & \mathbf{i} f\overline{b} & \mathbf{i} f\overline{b} & -\mathbf{i} f\overline{b} & -\mathbf{i} f\overline{b} & -1 \\
 1 & 1 & -a\overline{b} & a\overline{b} & -1 & -\mathbf{i} e\overline{b} & -1 & d f\overline{b^2} & -d f\overline{b^2} & \mathbf{i} f g\overline{b^2} & -\mathbf{i} f g\overline{b^2} & \mathbf{i} e\overline{b} \\
 1 & 1 & -1 & -1 & \mathbf{i} f\overline{b} & e f\overline{b^2} & \mathbf{i} f\overline{b} & -\mathbf{i} f\overline{b} & -\mathbf{i} f\overline{b} & -\mathbf{i} f g\overline{b^2} & \mathbf{i} f g\overline{b^2} & -e f\overline{b^2} \\
 1 & 1 & -1 & -1 & -\mathbf{i} f\overline{b} & -e f\overline{b^2} & -\mathbf{i} f\overline{b} & -d f\overline{b^2} & d f\overline{b^2} & \mathbf{i} f\overline{b} & \mathbf{i} f\overline{b} & e f\overline{b^2} \\
 1 & -1 & \mathbf{i} b & \mathbf{i} b & -\mathbf{i} b & e & -\mathbf{i} b & -\mathbf{i} d f\overline{b} & \mathbf{i} d f\overline{b} & f g\overline{b} & -f g\overline{b} & -e \\
 1 & -1 & \mathbf{i} a & -\mathbf{i} a & c f\overline{b} & -f & -c f\overline{b} & f & f & -f g\overline{b} & f g\overline{b} & -f \\
 1 & -1 & \mathbf{i} a & -\mathbf{i} a & -c f\overline{b} & f & c f\overline{b} & \mathbf{i} d f\overline{b} & -\mathbf{i} d f\overline{b} & -f & -f & f \\
 1 & -1 & -\mathbf{i} a & \mathbf{i} a & \mathbf{i} b & -\mathbf{i} b & \mathbf{i} b & -d & d & \mathbf{i} g & -\mathbf{i} g & -\mathbf{i} b \\
 1 & -1 & -\mathbf{i} a & \mathbf{i} a & -\mathbf{i} c & -e & \mathbf{i} c & d & -d & -\mathbf{i} g & \mathbf{i} g & e \\
 1 & -1 & -\mathbf{i} b & -\mathbf{i} b & \mathbf{i} c & \mathbf{i} b & -\mathbf{i} c & -f & -f & f & f & \mathbf{i} b
\end{array}
\right],\]
\[B_{12H}^{(6)}(a,b,c,d,e,f)=\left[
\begin{array}{rrrrrrrrrrrrr}
 1 & 1 & 1 & 1 & 1 & 1 & 1 & 1 & 1 & 1 & 1 & 1 \\
 1 & 1 & 1 & 1 & \mathbf{i} f & \mathbf{i} e & -1 & -1 & -1 & -1 & -\mathbf{i} e & -\mathbf{i} f \\
 1 & 1 & b & -b & -1 & -1 & 1 & 1 & \mathbf{i} d & -\mathbf{i} d & -1 & -1 \\
 1 & 1 & -b & b & -1 & -\mathbf{i} e & -1 & -1 & 1 & 1 & \mathbf{i} e & -1 \\
 1 & 1 & -1 & -1 & 1 & -1 & \mathbf{i} c & -\mathbf{i} c & -\mathbf{i} d & \mathbf{i} d & -1 & 1 \\
 1 & 1 & -1 & -1 & -\mathbf{i} f & 1 & -\mathbf{i} c & \mathbf{i} c & -1 & -1 & 1 & \mathbf{i} f \\
 1 & -1 & \mathbf{i} a & \mathbf{i} a & a f & a e & -a c & a c & -\mathbf{i} a & -\mathbf{i} a & -a e & -a f \\
 1 & -1 & \mathbf{i} a b & -\mathbf{i} a b & a f & -a e & -\mathbf{i} a & -\mathbf{i} a & \mathbf{i} a & \mathbf{i} a & a e & -a f \\
 1 & -1 & \mathbf{i} a b & -\mathbf{i} a b & -a f & a e & a c & -a c & a d & -a d & -a e & a f \\
 1 & -1 & -\mathbf{i} a b & \mathbf{i} a b & \mathbf{i} a & -\mathbf{i} a & a c & -a c & -a d & a d & -\mathbf{i} a & \mathbf{i} a \\
 1 & -1 & -\mathbf{i} a b & \mathbf{i} a b & -\mathbf{i} a & -a e & \mathbf{i} a & \mathbf{i} a & a d & -a d & a e & -\mathbf{i} a \\
 1 & -1 & -\mathbf{i} a & -\mathbf{i} a & -a f & \mathbf{i} a & -a c & a c & -a d & a d & \mathbf{i} a & a f
\end{array}
\right],\]
\[B_{12I}^{(4)}(a,b,c,d)=\left[
\begin{array}{rrrrrrrrrrrr}
 1 & 1 & 1 & 1 & 1 & 1 & 1 & 1 & 1 & 1 & 1 & 1 \\
 1 & 1 & 1 & 1 & \mathbf{i} & \mathbf{i} & -1 & -1 & -1 & -1 & -\mathbf{i} & -\mathbf{i} \\
 1 & 1 & a b & -a b & -1 & -\mathbf{i} & b c & \mathbf{i} b c & -b c & -\mathbf{i} b c & \mathbf{i} & -1 \\
 1 & 1 & -a b & a b & -\mathbf{i} & -1 & -\mathbf{i} b c & -b c & \mathbf{i} b c & b c & -1 & \mathbf{i} \\
 1 & 1 & -1 & -1 & b c & b c & -b c & -\mathbf{i} b c & -\mathbf{i} b c & -b c & \mathbf{i} b c & \mathbf{i} b c \\
 1 & 1 & -1 & -1 & -b c & -b c & \mathbf{i} b c & b c & b c & \mathbf{i} b c & -\mathbf{i} b c & -\mathbf{i} b c \\
 1 & -1 & \mathbf{i}\overline{b} & \mathbf{i}\overline{b} & \overline{b} & -\mathbf{i}\overline{b} & -c & c & \mathbf{i} c & -\mathbf{i} c & -\overline{b} & -\mathbf{i}\overline{b} \\
 1 & -1 & \mathbf{i} a & -\mathbf{i} a & \overline{b^2 d} & -\overline{b^2 d} & -\mathbf{i} c & \mathbf{i} c\overline{b d} & -\mathbf{i} c\overline{b d} & \mathbf{i} c & \overline{b} & -\overline{b} \\
 1 & -1 & \mathbf{i} a & -\mathbf{i} a & -\overline{b^2 d} & \overline{b^2 d} & c & -\mathbf{i} c\overline{b d} & \mathbf{i} c\overline{b d} & -c & -\mathbf{i}\overline{b} & \mathbf{i}\overline{b} \\
 1 & -1 & -\mathbf{i} a & \mathbf{i} a & -\overline{b} & \overline{b} & -b c d & -c & c & b c d & d & -d \\
 1 & -1 & -\mathbf{i} a & \mathbf{i} a & -\mathbf{i}\overline{b} & \mathbf{i}\overline{b} & b c d & \mathbf{i} c & -\mathbf{i} c & -b c d & -d & d \\
 1 & -1 & -\mathbf{i}\overline{b} & -\mathbf{i}\overline{b} & \mathbf{i}\overline{b} & -\overline{b} & \mathbf{i} c & -\mathbf{i} c & -c & c & \mathbf{i}\overline{b} & \overline{b}
\end{array}
\right],\]
\[B_{12J}^{(4)}(a,b,c,d)=\left[
\begin{array}{rrrrrrrrrrrr}
 1 & 1 & 1 & 1 & 1 & 1 & 1 & 1 & 1 & 1 & 1 & 1 \\
 1 & 1 & 1 & 1 & \mathbf{i} c\overline{b} & \mathbf{i} b\overline{c} & -1 & -1 & -1 & -1 & -\mathbf{i} b\overline{c} & -\mathbf{i} c\overline{b} \\
 1 & 1 & a b & -a b & -1 & -\mathbf{i} b\overline{c} & b d & \mathbf{i} b d & -b d & -\mathbf{i} b d & \mathbf{i} b\overline{c} & -1 \\
 1 & 1 & -a b & a b & -\mathbf{i} c\overline{b} & -1 & -\mathbf{i} c d & -c d & \mathbf{i} c d & c d & -1 & \mathbf{i} c\overline{b} \\
 1 & 1 & -1 & -1 & c d & b d & -b d & -\mathbf{i} b d & -\mathbf{i} c d & -c d & \mathbf{i} b d & \mathbf{i} c d \\
 1 & 1 & -1 & -1 & -c d & -b d & \mathbf{i} c d & c d & b d & \mathbf{i} b d & -\mathbf{i} b d & -\mathbf{i} c d \\
 1 & -1 & \mathbf{i}\overline{b} & \mathbf{i}\overline{b} & c\overline{b^2} & -\mathbf{i}\overline{b} & -c d\overline{b} & \mathbf{i} c d\overline{b} & c d\overline{b} & -\mathbf{i} c d\overline{b} & -\mathbf{i}\overline{b} & -c\overline{b^2} \\
 1 & -1 & \mathbf{i} a & -\mathbf{i} a & c d\overline{b} & -d & -\mathbf{i} d & d & -c d\overline{b} & \mathbf{i} c d\overline{b} & \mathbf{i} d & -\mathbf{i} c d\overline{b} \\
 1 & -1 & \mathbf{i} a & -\mathbf{i} a & -c d\overline{b} & d & c d\overline{b} & -\mathbf{i} c d\overline{b} & \mathbf{i} d & -d & -\mathbf{i} d & \mathbf{i} c d\overline{b} \\
 1 & -1 & -\mathbf{i} a & \mathbf{i} a & -c\overline{b^2} & \overline{c} & -\overline{b} & \overline{b} & -\overline{b} & \overline{b} & -\overline{c} & c\overline{b^2} \\
 1 & -1 & -\mathbf{i} a & \mathbf{i} a & -\mathbf{i}\overline{b} & \mathbf{i}\overline{b} & \overline{b} & -\overline{b} & \overline{b} & -\overline{b} & \mathbf{i}\overline{b} & -\mathbf{i}\overline{b} \\
 1 & -1 & -\mathbf{i}\overline{b} & -\mathbf{i}\overline{b} & \mathbf{i}\overline{b} & -\overline{c} & \mathbf{i} d & -d & -\mathbf{i} d & d & \overline{c} & \mathbf{i}\overline{b}
\end{array}
\right],\]
\[B_{12K}^{(4)}(a,b,c,d)=\left[
\begin{array}{rrrrrrrrrrrr}
 1 & 1 & 1 & 1 & 1 & 1 & 1 & 1 & 1 & 1 & 1 & 1 \\
 1 & 1 & 1 & 1 & \mathbf{i}\overline{b c} & \mathbf{i} b c & -1 & -1 & -1 & -1 & -\mathbf{i} b c & -\mathbf{i}\overline{b c} \\
 1 & 1 & a b & -a b & -1 & -\mathbf{i} b c & b d & \mathbf{i} b d & -b d & -\mathbf{i} b d & \mathbf{i} b c & -1 \\
 1 & 1 & -a b & a b & -\mathbf{i}\overline{b c} & -1 & -\mathbf{i} b^2 c d & -b^2 c d & \mathbf{i} b^2 c d & b^2 c d & -1 & \mathbf{i}\overline{b c} \\
 1 & 1 & -1 & -1 & b d & -b^2 c d & -b d & b^2 c d & -\mathbf{i} b^2 c d & \mathbf{i} b d & \mathbf{i} b^2 c d & -\mathbf{i} b d \\
 1 & 1 & -1 & -1 & -b d & b^2 c d & \mathbf{i} b^2 c d & -\mathbf{i} b d & b d & -b^2 c d & -\mathbf{i} b^2 c d & \mathbf{i} b d \\
 1 & -1 & \mathbf{i}\overline{b} & \mathbf{i}\overline{b} & \overline{b^2 c} & -\mathbf{i}\overline{b} & -b c d & \mathbf{i} b c d & b c d & -\mathbf{i} b c d & -\mathbf{i}\overline{b} & -\overline{b^2 c} \\
 1 & -1 & \mathbf{i} a & -\mathbf{i} a & d & b c d & -\mathbf{i} d & -\mathbf{i} b c d & -b c d & -d & \mathbf{i} b c d & \mathbf{i} d \\
 1 & -1 & \mathbf{i} a & -\mathbf{i} a & -d & -b c d & b c d & d & \mathbf{i} d & \mathbf{i} b c d & -\mathbf{i} b c d & -\mathbf{i} d \\
 1 & -1 & -\mathbf{i} a & \mathbf{i} a & -\overline{b^2 c} & c & -\overline{b} & \overline{b} & -\overline{b} & \overline{b} & -c & \overline{b^2 c} \\
 1 & -1 & -\mathbf{i} a & \mathbf{i} a & -\mathbf{i}\overline{b} & \mathbf{i}\overline{b} & \overline{b} & -\overline{b} & \overline{b} & -\overline{b} & \mathbf{i}\overline{b} & -\mathbf{i}\overline{b} \\
 1 & -1 & -\mathbf{i}\overline{b} & -\mathbf{i}\overline{b} & \mathbf{i}\overline{b} & -c & \mathbf{i} d & -d & -\mathbf{i} d & d & c & \mathbf{i}\overline{b}
\end{array}
\right],\]
\[B_{12L}^{(5)}(a,b,c,d,e)=\left[
\begin{array}{rrrrrrrrrrrrrr}
 1 & 1 & 1 & 1 & 1 & 1 & 1 & 1 & 1 & 1 & 1 & 1 \\
 1 & 1 & 1 & 1 & \mathbf{i} b\overline{c} & \mathbf{i} c\overline{b} & -1 & -1 & -1 & -1 & -\mathbf{i} c\overline{b} & -\mathbf{i} b\overline{c} \\
 1 & 1 & a\overline{b} & -a\overline{b} & -1 & -\mathbf{i} c\overline{b} & d\overline{c} & \mathbf{i} d\overline{b} & -d\overline{c} & -\mathbf{i} d\overline{b} & \mathbf{i} c\overline{b} & -1 \\
 1 & 1 & -a\overline{b} & a\overline{b} & -\mathbf{i} b\overline{c} & -1 & \mathbf{i} e\overline{b} & c e\overline{b^2} & -\mathbf{i} e\overline{b} & -c e\overline{b^2} & -1 & \mathbf{i} b\overline{c} \\
 1 & 1 & -1 & -1 & d\overline{c} & -c e\overline{b^2} & -d\overline{c} & -\mathbf{i} d\overline{b} & \mathbf{i} e\overline{b} & c e\overline{b^2} & \mathbf{i} d\overline{b} & -\mathbf{i} e\overline{b} \\
 1 & 1 & -1 & -1 & -d\overline{c} & c e\overline{b^2} & -\mathbf{i} e\overline{b} & -c e\overline{b^2} & d\overline{c} & \mathbf{i} d\overline{b} & -\mathbf{i} d\overline{b} & \mathbf{i} e\overline{b} \\
 1 & -1 & \mathbf{i} b & \mathbf{i} b & b^2\overline{c} & -\mathbf{i} b & e & -\mathbf{i} c e\overline{b} & -e & \mathbf{i} c e\overline{b} & -\mathbf{i} b & -b^2\overline{c} \\
 1 & -1 & \mathbf{i} a & -\mathbf{i} a & \mathbf{i} b d\overline{c} & \mathbf{i} c e\overline{b} & -\mathbf{i} b d\overline{c} & d & e & -\mathbf{i} c e\overline{b} & -d & -e \\
\end{array}
\right],\]
\[B_{12M}^{(3)}(a,b,c)=\left[
\begin{array}{rrrrrrrrrrrr}
 1 & 1 & 1 & 1 & 1 & 1 & 1 & 1 & 1 & 1 & 1 & 1 \\
 1 & 1 & a b & b c & \mathbf{i} a b & \mathbf{i} b c & -b c & -1 & -1 & -a b & -\mathbf{i} b c & -\mathbf{i} a b \\
 1 & 1 & a b & -b c & -a b & -\mathbf{i} b c & b c & \mathbf{i} & -1 & -\mathbf{i} & \mathbf{i} b c & -1 \\
 1 & 1 & -a b & \mathbf{i} & a b & -\mathbf{i} b c & \mathbf{i} b c & -\mathbf{i} & -\mathbf{i} & -1 & -1 & \mathbf{i} \\
 1 & 1 & -a b & -1 & -\mathbf{i} a b & \mathbf{i} b c & -\mathbf{i} b c & -1 & \mathbf{i} & a b & -\mathbf{i} & \mathbf{i} a b \\
 1 & 1 & -1 & -\mathbf{i} & -1 & -1 & -1 & 1 & 1 & \mathbf{i} & \mathbf{i} & -\mathbf{i} \\
 1 & -1 & \mathbf{i}\overline{b} & \overline{b} & \mathbf{i}\overline{b} & -\mathbf{i}\overline{b} & -\mathbf{i}\overline{b} & -\mathbf{i}\overline{b} & \mathbf{i}\overline{b} & \overline{b} & -\overline{b} & -\overline{b} \\
 1 & -1 & \mathbf{i} a & -\overline{b} & a & -c & c & \mathbf{i}\overline{b} & \overline{b} & -a & -\mathbf{i}\overline{b} & -\mathbf{i} a \\
 1 & -1 & \mathbf{i} a & -\mathbf{i}\overline{b} & -\mathbf{i} a & c & -c & \overline{b} & -\overline{b} & -\overline{b} & \overline{b} & \mathbf{i}\overline{b} \\
 1 & -1 & -\mathbf{i} a & c & -a & -c & \mathbf{i} c & \mathbf{i}\overline{b} & -\mathbf{i}\overline{b} & a & -\mathbf{i} c & \mathbf{i} a \\
 1 & -1 & -\mathbf{i}\overline{b} & \mathbf{i}\overline{b} & -\mathbf{i}\overline{b} & \mathbf{i}\overline{b} & \mathbf{i}\overline{b} & -\mathbf{i}\overline{b} & \mathbf{i}\overline{b} & -\mathbf{i}\overline{b} & \mathbf{i}\overline{b} & -\mathbf{i}\overline{b} \\
 1 & -1 & -\mathbf{i} a & -c & \mathbf{i} a & c & -\mathbf{i} c & -\overline{b} & -\mathbf{i}\overline{b} & \mathbf{i}\overline{b} & \mathbf{i} c & \overline{b}
\end{array}
\right],\]
\[B_{12N}^{(6)}(a,b,c,d,e,f)=\left[
\begin{array}{rrrrrrrrrrrr}
 1 & 1 & 1 & 1 & 1 & 1 & 1 & 1 & 1 & 1 & 1 & 1 \\
 1 & 1 & a b & b c & \mathbf{i} & \mathbf{i} & -b c & -1 & -1 & -a b & -\mathbf{i} & -\mathbf{i} \\
 1 & 1 & \mathbf{i} a b & \mathbf{i} & -\mathbf{i} & -\mathbf{i} & \mathbf{i} & \mathbf{i} b f & -\mathbf{i} b f & -\mathbf{i} a b & -1 & -1 \\
 1 & 1 & -1 & -b c & b d & -b d & b c & -\mathbf{i} b f & \mathbf{i} b f & -1 & b e & -b e \\
 1 & 1 & -a b & -\mathbf{i} & -1 & -1 & -\mathbf{i} & \mathbf{i} & \mathbf{i} & a b & -b e & b e \\
 1 & 1 & -\mathbf{i} a b & -1 & -b d & b d & -1 & -\mathbf{i} & -\mathbf{i} & \mathbf{i} a b & \mathbf{i} & \mathbf{i} \\
 1 & -1 & \overline{b} & \mathbf{i} c & \mathbf{i} d & -\mathbf{i} d & -\mathbf{i} c & -\mathbf{i} f & \mathbf{i} f & \overline{b} & -\overline{b} & -\overline{b} \\
 1 & -1 & a & -\overline{b} & -\mathbf{i}\overline{b} & -\mathbf{i}\overline{b} & -\overline{b} & \mathbf{i}\overline{b} & \mathbf{i}\overline{b} & -a & \overline{b} & \overline{b} \\
 1 & -1 & \mathbf{i}\overline{b} & -\mathbf{i} c & d & -d & \mathbf{i} c & -\mathbf{i}\overline{b} & -\mathbf{i}\overline{b} & \mathbf{i}\overline{b} & -e & e \\
 1 & -1 & -\overline{b} & \overline{b} & -d & d & \overline{b} & -f & f & -\overline{b} & \mathbf{i} e & -\mathbf{i} e \\
 1 & -1 & -a & c & -\mathbf{i} d & \mathbf{i} d & -c & f & -f & a & e & -e \\
 1 & -1 & -\mathbf{i}\overline{b} & -c & \mathbf{i}\overline{b} & \mathbf{i}\overline{b} & c & \mathbf{i} f & -\mathbf{i} f & -\mathbf{i}\overline{b} & -\mathbf{i} e & \mathbf{i} e\end{array}
\right],\]
\[B_{12O}^{(3)}(a,b,c)=\left[
\begin{array}{rrrrrrrrrrrr}
 1 & 1 & 1 & 1 & 1 & 1 & 1 & 1 & 1 & 1 & 1 & 1 \\
 1 & 1 & \overline{a b} & a b & \mathbf{i} a b & \mathbf{i}\overline{a b} & -a b & -1 & -\overline{a b} & -1 & -\mathbf{i} a b & -\mathbf{i}\overline{a b} \\
 1 & 1 & \mathbf{i} & \mathbf{i} a b & -\mathbf{i} a b & -\mathbf{i} & a b & \mathbf{i} a c & -1 & -\mathbf{i} a c & -a b & -1 \\
 1 & 1 & -1 & -a b & a b & -1 & -\mathbf{i} a b & -\mathbf{i} a c & \mathbf{i} & \mathbf{i} a c & \mathbf{i} a b & -\mathbf{i} \\
 1 & 1 & -\overline{a b} & -\mathbf{i} a b & -1 & \mathbf{i} & \mathbf{i} a b & -\mathbf{i} & \overline{a b} & -\mathbf{i} & -1 & \mathbf{i} \\
 1 & 1 & -\mathbf{i} & -1 & -a b & -\mathbf{i}\overline{a b} & -1 & \mathbf{i} & -\mathbf{i} & \mathbf{i} & a b & \mathbf{i}\overline{a b} \\
 1 & -1 & b & -b & a b^2 & -\overline{a} & -b & \mathbf{i} c & b & -\mathbf{i} c & -a b^2 & \overline{a} \\
 1 & -1 & \overline{a} & -a b^2 & -\mathbf{i} b & \mathbf{i} b & a b^2 & -\mathbf{i} c & -\overline{a} & \mathbf{i} c & -\mathbf{i} b & \mathbf{i} b \\
 1 & -1 & \mathbf{i} b & -\mathbf{i} a b^2 & -a b^2 & b & -a b^2 & a b^2 & -b & a b^2 & \mathbf{i} a b^2 & -\mathbf{i} b \\
 1 & -1 & -b & a b^2 & -\mathbf{i} a b^2 & -\mathbf{i} b & \mathbf{i} a b^2 & -a b^2 & \mathbf{i} b & -a b^2 & a b^2 & b \\
 1 & -1 & -\overline{a} & \mathbf{i} a b^2 & \mathbf{i} a b^2 & \overline{a} & -\mathbf{i} a b^2 & c & \overline{a} & -c & -\mathbf{i} a b^2 & -\overline{a} \\
 1 & -1 & -\mathbf{i} b & b & \mathbf{i} b & -b & b & -c & -\mathbf{i} b & c & \mathbf{i} b & -b
\end{array}
\right],\]
\[B_{12P}^{(4)}(a,b,c,d)=\left[
\begin{array}{rrrrrrrrrrrr}
 1 & 1 & 1 & 1 & 1 & 1 & 1 & 1 & 1 & 1 & 1 & 1 \\
 1 & 1 & c d & a c & \mathbf{i} a c & \mathbf{i} c d & -a c & -1 & -c d & -1 & -\mathbf{i} c d & -\mathbf{i} a c \\
 1 & 1 & \mathbf{i} & \mathbf{i} a c & -\mathbf{i} a c & -\mathbf{i} & a c & \mathbf{i} b c & -1 & -\mathbf{i} b c & -1 & -a c \\
 1 & 1 & -c d & -1 & \mathbf{i} & -\mathbf{i} c d & -\mathbf{i} & -\mathbf{i} b c & \mathbf{i} c d & \mathbf{i} b c & c d & -1 \\
 1 & 1 & -1 & -\mathbf{i} a c & -1 & \mathbf{i} & \mathbf{i} & -\mathbf{i} & c d & -\mathbf{i} & -c d & \mathbf{i} a c \\
 1 & 1 & -\mathbf{i} & -a c & -\mathbf{i} & -1 & -1 & \mathbf{i} & -\mathbf{i} c d & \mathbf{i} & \mathbf{i} c d & a c \\
 1 & -1 & \overline{c} & -\overline{c} & -\overline{c} & \overline{c} & -\overline{c} & \mathbf{i} b & \overline{c} & -\mathbf{i} b & \overline{c} & -\overline{c} \\
 1 & -1 & d & -a & -\mathbf{i} a & \mathbf{i} d & a & -\mathbf{i} b & -d & \mathbf{i} b & -\mathbf{i} d & \mathbf{i} a \\
 1 & -1 & \mathbf{i}\overline{c} & -\mathbf{i} a & \mathbf{i} a & -\mathbf{i}\overline{c} & -a & \overline{c} & -\overline{c} & \overline{c} & -\overline{c} & a \\
 1 & -1 & -d & \overline{c} & -\mathbf{i}\overline{c} & -\mathbf{i} d & \mathbf{i}\overline{c} & -\overline{c} & \mathbf{i} d & -\overline{c} & d & \overline{c} \\
 1 & -1 & -\overline{c} & \mathbf{i} a & \overline{c} & \mathbf{i}\overline{c} & -\mathbf{i}\overline{c} & b & d & -b & -d & -\mathbf{i} a \\
 1 & -1 & -\mathbf{i}\overline{c} & a & \mathbf{i}\overline{c} & -\overline{c} & \overline{c} & -b & -\mathbf{i} d & b & \mathbf{i} d & -a
\end{array}
\right],\]
\[B_{12Q}^{(4)}(a,b,c,d)=\left[
\begin{array}{rrrrrrrrrrrrrr}
 1 & 1 & 1 & 1 & 1 & 1 & 1 & 1 & 1 & 1 & 1 & 1 \\
 1 & 1 & b & b & b & b & -b & -b & -b & -1 & -1 & -b \\
 1 & 1 & \mathbf{i} b & -b & -b & -\mathbf{i} b & b & b & \mathbf{i} b & -1 & -1 & -\mathbf{i} b \\
 1 & 1 & -1 & \mathbf{i} & -\mathbf{i} & -1 & \mathbf{i} & -\mathbf{i} & -1 & 1 & 1 & -1 \\
 1 & 1 & -b & -\mathbf{i} & -1 & \mathbf{i} b & -\mathbf{i} & -1 & -\mathbf{i} b & \mathbf{i} & \mathbf{i} & b \\
 1 & 1 & -\mathbf{i} b & -1 & \mathbf{i} & -b & -1 & \mathbf{i} & b & -\mathbf{i} & -\mathbf{i} & \mathbf{i} b \\
 1 & -1 & a & \mathbf{i} c & -a & -\mathbf{i} c & \mathbf{i} a & -c & -\mathbf{i} a & d & -d & c \\
 1 & -1 & c & -c & -\mathbf{i} c & \mathbf{i} c & -\mathbf{i} c & c & -c & -\mathbf{i} d & \mathbf{i} d & \mathbf{i} c \\
 1 & -1 & \mathbf{i} a & c & -\mathbf{i} a & -c & -a & \mathbf{i} c & a & -d & d & -\mathbf{i} c \\
 1 & -1 & -a & a & \mathbf{i} a & -\mathbf{i} a & -\mathbf{i} a & a & -a & \mathbf{i} d & -\mathbf{i} d & \mathbf{i} a \\
 1 & -1 & -c & -\mathbf{i} c & \mathbf{i} c & c & \mathbf{i} c & -\mathbf{i} c & c & -\mathbf{i} d & \mathbf{i} d & -c \\
 1 & -1 & -\mathbf{i} a & -a & a & \mathbf{i} a & a & -a & \mathbf{i} a & \mathbf{i} d & -\mathbf{i} d & -\mathbf{i} a
\end{array}
\right].\] \normalsize Table \ref{table2} summarizes the relations
between $\BH(4,12)$ matrices and parametric families by showing which
$\mathrm{ACT}$-classes are included in each of the $23$ parametric
families. An $\mathrm{ACT}$-class number is underlined if the
$\mathrm{ACT}$-class belongs to only one of the parametric
families listed in the table.  \tiny

\begin{table}[h]
\centering 
\caption{Summary of the parametric families of order $12$}\label{table2} 
\begin{tabular}{ll} 
\hline
Family & $\mathrm{ACT}$-classes\\
\hline
 $H_{12B}^{(10)}$ & $1,2,3,4,5,6,7,9,10,11,\underline{12},13,14,15,16,17,\underline{18},19,21,$\\ & $22,23,28,\underline{29},\underline{45},\underline{46},47,\underline{48},\underline{49},59,\underline{62},\underline{65},\underline{66},\underline{70},78,\underline{79},81$ \\
 $H_{12C}^{(8)}$ & $1,2,3,4,5,6,7,8,9,10,11,14,15,16,17,19,\underline{20},21,22,23,24,25,26,$\\ & $\underline{27},28,32,33,34,37,38,39,40,\underline{50},51,\underline{52},\underline{53},\underline{55},\underline{57},58,60,61,63,64$
   \\
 $D_{12}^{(7)}$ & $103,104,105,106,107,108,109,\underline{110},111,112,$\\ & $\underline{113},\underline{114},\underline{115},117,127,\underline{128},129,130,\underline{131},138$ \\
 $X_{12}^{(7)}$ & $\underline{44},75,\underline{76},107,116,118,\underline{120},121,122,123,124,$\\
\hline
\end{tabular}
\end{table}
\normalsize
Table \ref{table3} summarizes the properties of $\BH(4,12)$ matrices.
\tiny

\begin{table}[h]
\setlength{\tabcolsep}{3pt} 
\centering 
\caption{Summary of the $\BH(4,12)$ matrices}\label{table3} 
\begin{tabular}{cclccccccc} 
\hline
ACT & Equiv. & Family, & ACT & HBS & Auto. & Defect & Orbit & $\mathbb{Z}_4$ & Invariant\\
class & classes & coordinates & & & order & & & rank & $(4\times 4,5\times 5)$\\
\hline
 1 & $1$ & $H_{12B}^{(10)}(1,-1,-1,-1,1,-1,1,1,-1,1)$ & YYY & YNY & 380160 & 55 & 10 & 10 & $(109890,205920)$ \\
 2 & $2,8$ & $H_{12B}^{(10)}(-1,-1,1,-1,-1,-1,-1,-\mathbf{i},1,1)$ & NYN & NNN & 5760 & 45 & 10 & 9 & $(86130,146400)$ \\
 3 & $3,23$ & $H_{12B}^{(10)}(-1,-1,1,-\mathbf{i},1,1,-1,-1,-\mathbf{i},1)$ & NYN & NNN & 256 & 35 & 10 & 8 & $(66578,106208)$ \\
 4 & $4$ & $H_{12B}^{(10)}(\mathbf{i},-1,-\mathbf{i},1,-1,1,1,-\mathbf{i},1,-\mathbf{i})$ & YYY & YNY & 192 & 31 & 10 & 8 & $(55554,86304)$ \\
 5 & $5,13$ & $H_{12B}^{(10)}(1,-\mathbf{i},-1,-\mathbf{i},-1,1,1,-\mathbf{i},-1,1)$ & NYN & NNN & 64 & 29 & 10 & 7 & $(56258,82848)$ \\
 6 & $6$ & $H_{12B}^{(10)}(\mathbf{i},1,-1,-\mathbf{i},-1,-1,-1,-1,\mathbf{i},\mathbf{i})$ & YYY & YNY & 192 & 37 & 10 & 8 & $(68082,105392)$ \\
\hline
\end{tabular}
\end{table}

\begin{table}[h]
\setlength{\tabcolsep}{3pt}. 
\centering 
\begin{tabular}{cclccccccc} 
\hline
ACT & Equiv. & Family, & ACT & HBS & Auto. & Defect & Orbit & $\mathbb{Z}_4$ & Invariant\\
class & classes & coordinates & & & order & & & rank & $(4\times 4,5\times 5)$\\
\hline
 56 & $76,112$ & $B_{12D}^{(4)}(\mathbf{i},\mathbf{i},\mathbf{i},1)$ & NYN & NNN & 8 & 15 & 4 & 6 & $(24670,28992)$ \\
 57 & $78,82$ & $H_{12C}^{(8)}(1,\mathbf{i},\mathbf{i},-\mathbf{i},-\mathbf{i},\mathbf{i},1,1)$ & NYN & NNN & 8 & 15 & 8 & 5 & $(34136,45692)$ \\
 58 & $83,121$ & $H_{12C}^{(8)}(\mathbf{i},\mathbf{i},-\mathbf{i},\mathbf{i},1,\mathbf{i},1,-\mathbf{i})$ & NYN & NNN & 16 & 15 & 8 & 6 & $(29822,44900)$ \\
 59 & $85,91$ & $H_{12B}^{(10)}(\mathbf{i},\mathbf{i},-\mathbf{i},-\mathbf{i},1,-1,-\mathbf{i},-\mathbf{i},\mathbf{i},1)$ & NYN & NNN & 128 & 23 & 10 & 5 & $(41618,56864)$ \\
 60 & $87,128$ & $H_{12C}^{(8)}(\mathbf{i},-1,\mathbf{i},-\mathbf{i},-\mathbf{i},1,-1,-\mathbf{i})$ & NYN & NNN & 32 & 13 & 8 & 5 & $(32006,50144)$ \\
 61 & $88,124$ & $H_{12C}^{(8)}(\mathbf{i},-1,\mathbf{i},\mathbf{i},-\mathbf{i},-\mathbf{i},\mathbf{i},-\mathbf{i})$ & NYN & NNN & 32 & 13 & 8 & 5 & $(29470,41112)$ \\
\hline
 \end{tabular}
\end{table}

\begin{table}[h]
\setlength{\tabcolsep}{3pt}. 
\centering 
\begin{tabular}{cclccccccc} 
\hline
ACT & Equiv. & Family, & ACT & HBS & Auto. & Defect & Orbit & $\mathbb{Z}_4$ & Invariant\\
class & classes & coordinates & & & order & & & rank & $(4\times 4,5\times 5)$\\
\hline

 111 & $165,218$ & $D_{12}^{(7)}(-1,-\mathbf{i},1,-\mathbf{i},\mathbf{i},\mathbf{i},\mathbf{i})$ & NYN & NYN & 192 & 25 & 7 & 6 & $(32358,73632)$ \\
 112 & $166,219$ & $D_{12}^{(7)}(\mathbf{i},\mathbf{i},\mathbf{i},1,\mathbf{i},\mathbf{i},-\mathbf{i})$ & NYN & NYN & 16 & 17 & 7 & 6 & $(27406,66632)$ \\
 113 & $167,303$ & $D_{12}^{(7)}(1,1,-1,-\mathbf{i},\mathbf{i},\mathbf{i},-1)$ & NYN & NYN & 32 & 13 & 7 & 5 & $(27070,67008)$ \\
 114 & $168,304$ & $D_{12}^{(7)}(\mathbf{i},-\mathbf{i},\mathbf{i},1,\mathbf{i},1,-1)$ & NYN & NYN & 48 & 13 & 7 & 5 & $(26946,64392)$ \\
 115 & $169,307$ & $D_{12}^{(7)}(-1,\mathbf{i},-\mathbf{i},-1,\mathbf{i},1,-1)$ & NYN & NYN & 96 & 13 & 7 & 5 & $(26958,70512)$ \\
 116 & $170,196$ & $X_{12}^{(7)}(-1,-1,\mathbf{i},-1,1,\mathbf{i},\mathbf{i})$ & NYN & NYN & 64 & 19 & 7 & 6 & $(26870,50752)$ \\
\hline
\end{tabular}
\end{table}

\end{document}